\documentclass{amsart}
\usepackage{amsmath,amssymb,amsthm,amsfonts}

\DeclareMathSymbol{\twoheadrightarrow} {\mathrel}{AMSa}{"10}

\def\Q{{\mathbb Q}}

        \def\CC{{\mathfrak C}}

\def\Z{{\mathbb Z}}

\def\F{{\mathbb F}}

\def\ST{{\mathbf S}}

\def\A{{\mathbf A}}

\def\Gal{\mathrm{Gal}}

\def\Fr{\mathrm{Fr}}
\def\Frob{\mathrm{Frob}}

\def\End{\mathrm{End}}

\def\Aut{\mathrm{Aut}}

\def\Hom{\mathrm{Hom}}

\def\II{\mathrm{Id}}
                                    \def\TT{\mathrm{T}}

                                        \def\CC{{\mathcal C}}
                                         \def\RR{{\mathcal R}}

       \def\Oc{{\mathcal O}}

\def\fchar{\mathrm{char}}

\def\GL{\mathrm{GL}}

                                                   \def\Tr{\mathrm{Tr}}

                                                 \def\GSp{\mathfrak{Gp}}

\def\Sp{\mathrm{Sp}}

\def\Gp{\mathrm{Gp}}

        \def\res{\mathrm{res}}

        \def\K_a{\bar{K}}

\def\dim{\mathrm{dim}}

                                       \def\res{\mathrm{res}}

\def\P{{\mathbf P}}

\def\GG{{\mathfrak G}}

\def\K{{\mathcal{K}}}

\def\ZZ{{\mathfrak Z}}

 \def\ub{{\mathbf u}}

\def\TT{{\mathfrak T}}

\newtheorem{thm}{Theorem}[section]

\newtheorem{lem}[thm]{Lemma}

\newtheorem{cor}[thm]{Corollary}

\theoremstyle{definition}

\newtheorem{ex}[thm]{Example}

\newtheorem{rem}[thm]{Remark}

\newtheorem{sect}[thm]{}

\title[Endomorphism rings of reductions of elliptic curves]{Endomorphism rings of reductions of elliptic curves and abelian varieties}

\author[Yuri G.\ Zarhin]{Yuri G.\ Zarhin}
\thanks{This work was partially supported by a grant from the Simons Foundation (\#246625 to Yuri Zarkhin).}

\address{Department of Mathematics, Pennsylvania State University,
University Park, PA 16802, USA}

\email{zarhin\char`\@math.psu.edu}

\begin{document}

\begin{abstract}
Let $E$ be an elliptic curve without CM that is defined over a number field $K$.  For all but finitely many nonarchimedean places $v$ of $K$ there is the reduction $E(v)$ of $E$ at $v$ that is an elliptic curve over the residue field $k(v)$ at $v$. The set of $v$'s with ordinary $E(v)$ has density 1 (Serre).  For such $v$ the endomorphism ring $\End(E(v))$ of $E(v)$ is an order in an imaginary quadratic field.
%$L(v)=\End(E(v))\otimes \Q$.  
%We write ${\bf \Delta}(v)$ for the discriminant of $\End(E_{v})$.

We prove that for any pair of relatively prime positive integers $N$ and $M$ there are infinitely many nonarchimedean places $v$ of $K$ such that 
the {\sl discriminant}
 ${\bf \Delta(v)}$ of  $\End(E(v))$ is divisible by $N$ and the ratio $\frac{{\bf \Delta(v)}}{N}$ is relatively prime to $NM$. 
We also discuss similar questions for reductions of  abelian varieties.

The subject of this paper was inspired by an exercise in Serre's ``Abelian $\ell$-adic representations and elliptic curves" and  questions
of Mihran Papikian and Alina Cojocaru.
\end{abstract}

\maketitle

\section{Introduction}
\label{intro}

Let $K$ be a field, $\bar{K}$ its algebraic closure,
$\Gal(K)=\Aut(\bar{K}/K)$ the absolute Galois group of $K$. 
 Let $A$
be an abelian variety of positive dimension over $K$. We write $\End(A)$ for its endomorphism ring
and $\End^0(A)$ for the corresponding finite-dimensional semisimple $\Q$-algebra $\End(A)\otimes\Q$.
One may view  $\End(A)$ as an {\sl order} in $\End^0(A)$.

Let $n$
be a positive integer  that is {\sl not} divisible by $\fchar(K)$.
We write $A[n]$ for the kernel
of multiplication by $n$ in $A(\bar{K})$. It is well known that is a
finite Galois submodule of $A(\bar{K})$; if we forget about the
Galois action 
then the commutative group $A[n]$ is a free $\Z/n\Z$-module of rank
$2\dim(A)$. If $\ell$ is a prime different from $\fchar(K)$ then we write $T_{\ell}(A)$ for the $\Z_{\ell}$-Tate module of $A$ that is defined as projective limit of commutative groups (Galois modules) $A_{\ell^i}$ where the transition map $A_{\ell^{i+1}} \to A_{\ell^i}$ is multiplication by $\ell$. It is well known that $T_{\ell}(A)$ is a free $\Z_{\ell}$-module of rank $2\dim(A)$ provided with continuous Galois action
$$\rho_{\ell,A}\to \Aut_{\Z_{\ell}}(T_{\ell}(A)).$$
In particular, $T_{\ell}(A)$ carries the natural structure of $\Gal(K)$-module. On the other hand, 
%if $\End(A)$ is the ring of $K$-endomorphisms of $A$ then its
the natural action of $\End(A)$ on $A_n$ gives rise to the embedding
$$\End(A)\otimes \Z/n\Z \hookrightarrow \End_{\Gal(K)}(A_n).$$
If (as above) we put $n=\ell^i$ then these embedding are glueing together to the embedding of $\Z_{\ell}$-algebras
$$\End(A)\otimes{\Z_{\ell}}\hookrightarrow \End_{\Gal(K)}(T_{\ell})(A)). \eqno(*)$$
Tate \cite{Tate1,Tate2} conjectured that if $K$ is finitely generated then the embedding (*) is actually a bijection and proved it when $K$ is a finite field. The case when $\fchar(K)>2$ was done by the author \cite{ZarhinIzv75,ZarhinMatZametki76}, the case when $\fchar(K)=0$ by Faltings \cite{Faltings1,Faltings2} and the case when $\fchar(K)=2$ by Mori \cite{MB} (see also \cite{ZarhinParshin,ZarhinCEJM2014,ZarhinMori}).

Now let us consider the $2\dim(A)$-dimensional $\Q_{\ell}$-vector space
$$V_{\ell}(A)=T_{\ell}(A)\otimes_{\Z_{\ell}}\Q_{\ell}$$
and identify $T_{\ell}(A)$ with the $\Z_{\ell}$-lattice 
$$T_{\ell}(A)\otimes 1 \subset  T_{\ell}(A)\otimes_{\Z_{\ell}}\Q_{\ell}=V_{\ell}(A).$$
This allows us to identify $\Aut_{\Z_{\ell}}(T_{\ell}(A))$ with the (compact) subgroup of
 $\Aut_{\Q_{\ell}}(V_{\ell}(A))$  that consists of all automorphisms that leave invariant  $T_{\ell}(A)$
and consider $\rho_{\ell,A}$ as the $\ell$-adic representation
$$\rho_{\ell,A}:\Gal(K)\to \Aut_{\Z_{\ell}}(T_{\ell}(A)\subset \Aut_{\Q_{\ell}}(V_{\ell}(A)).$$
(By definition, $T_{\ell}(A)$ is a Galois-stable $\Z_{\ell}$-lattice in $V_{\ell}(A)$.) We write $G_{\ell,A}$ for the image
$$G_{\ell,A}:=\rho_{\ell,A}(\Gal(K))\subset \Aut_{\Z_{\ell}}(T_{\ell}(A))\subset \Aut_{\Q_{\ell}}(V_{\ell}(A)).$$
It is known \cite{Serre} that $G_{\ell,A}$ is a compact $\ell$-adic subgroup of $\Aut_{\Q_{\ell}}(V_{\ell}(A))$.
Extending the embedding (*) by $\Q_{\ell}$-linearity, we get the embedding of $\Q_{\ell}$-algebras
$$\End^0(A)\otimes_{\Q}\Q_{\ell}=\End(A)\otimes{\Q_{\ell}}\hookrightarrow \End_{\Gal(K)}(V_{\ell})(A))\subset \End_{\Q_{\ell}}(V_{\ell}(A)). \eqno(**)$$
When $K$ is finitely generated then the $\Gal(K)$-module $V_{\ell}(A)$ is semisimple: the case of finite fields was done by A.Weil \cite{Mumford},
the case when $\fchar(K)>2$ was done by the author \cite{ZarhinIzv75,ZarhinMatZametki76}, the case when $\fchar(K)=0$ by Faltings  \cite{Faltings1,Faltings2} and the case when $\fchar(K)=2$ by Mori \cite{MB} (see also \cite{ZarhinCEJM2014}). The semisimplicity of the Galois module $V_{\ell}(A)$ means that the $G_{\ell,A}$-module $V_{\ell}(A)$ is semisimple.

\begin{ex}[see \cite{Tate2}]
\label{finite}
Let $k$ be a finite field and 
$$\sigma_k : \bar{k} \to \bar{k}, \ x \mapsto x^{\#(k)}$$
be the Frobenius automorphism of its algebraic closure. Then $\bar{k}$ is a topological generator of
$\Gal(K)$. If $B$ is an abelian variety over $k$ of positive dimension then by Tate's {\sl theorem on homomorphisms}
$$\End(B)\otimes \Z_{\ell}=\End_{\Gal(k)}(T_{\ell}(B))$$
coincides with the centralizer $\End_{\sigma_k}(T_{\ell}(B))$ of $\sigma_k$ in $\End_{\Z_{\ell}}(T_{\ell}(B)$. In addition, $\sigma_k$ induces a semisimple (diagonalizable over $\bar{\Q_{\ell}}$) linear operator $\Fr_B$ in $V_{\ell}(B)$. The ring $\End(B)$ is commutative if and only if the characteristic polynomial
$$P_{\Fr_B}(t)=\det(t \II-\sigma_k, V_{\ell}(B))\in \Q_{\ell}[t]$$
has no multiple roots. (Actually this polynomial has integer coefficients and does not depend on a choice of $\ell \ne \fchar(k)$.)
\end{ex}

Let $\GG_{\ell,A} \subset \GL(V_{\ell}(A))$  be the Zariski closure of 
$$G_{\ell,A}\subset \Aut_{\Q_{\ell}}(V_{\ell}(A))=\GL(V_{\ell}(A))(\Q_{\ell})$$
in the general linear group $\GL(V_{\ell}(A))$  over $\Q_{\ell}$. By definition, 
$\GG_{\ell,A}$ is a linear $\Q_{\ell}$-algebraic subgroup of $\GL(V_{\ell}(A))$.
When $K$ is finitely generated, the semisimplicity of the $G_{\ell,A}$-module $V_{\ell}(A)$ means that (the identity component of) $\GG_{\ell,A}$ is a reductive algebraic group over $\Q_{\ell}$. If,  in addition, $\fchar(K)=0$ then
by a theorem of Bogomolov \cite{Bo1,Bo2,SerreR},  $G_{\ell,A}$ is an {\sl open} subgroup in $\GG_{\ell,A}(\Q_{\ell})$.  It is known \cite{SerreR} that the group $\GG_{\ell,A}$ is connected for one prime $\ell$ then it is connected for all primes.

\begin{sect}
\label{numberField}
Let $K$ be a number field. For all but finitely many nonarchimedean places $v$ of $K$ one may define the {\sl reduction} $A(v)$, which is an abelian variety of the same dimension as $A$ over the (finite) residue field $k(v)$ at $v$ \cite{SerreTate}. If $\ell$ does {\sl not} coincide with the residual characteristic of $v$ then 
each extension $\bar{v}$ of $v$ to $\bar{K}$ gives rise to an isomorphism of Tate modules $T_{\ell}(A(v))\cong T_{\ell}(A)$ that, in turn,  gives rise to the natural isomorphisms
 $$\End_{\Z_{\ell}}(T_{\ell}(A(v)))\cong \End_{\Z_{\ell}}(T_{\ell}(A))
, \ \Aut_{\Z_{\ell}}(T_{\ell}(A(v)))\cong \Aut_{\Z_{\ell}}(T_{\ell}(A)).$$
Under this isomorphism 
$$\Fr_{A(v)} \in \Aut_{\Z_{\ell}}(T_{\ell}(A(v)))\subset \End_{\Z_{\ell}}(T_{\ell}(A(v)))$$
corresponds to a certain element
$$\Frob_{\bar{v},A,\ell} \in G_{\ell,A}\subset  \Aut_{\Z_{\ell}}(T_{\ell}(A))\subset \End_{\Z_{\ell}}(T_{\ell}(A)) \subset \End_{\Q_{\ell}}(V_{\ell}(A))$$
that is called the {\sl Frobenius element} attached to $\bar{v}$ in $ G_{\ell,A}$. (All $\Frob_{\bar{v},A,\ell}$'s for a given $v$ constitute a conjugacy class in $G_{\ell,A}$.) This implies that the polynomial $P_{\Fr_{A(v)}}(t)$ coincides with the characteristic polynomial
$$P_{v,A}(t):=\det(t\II - \Frob_{\bar{v},A,\ell}, V_{\ell}(A)$$
of $\Frob_{\bar{v},A,\ell}$.
In particular, $\End(A(v))$ is commutative if and only if $P_{v,A}(t)$ has {\sl no} multiple roots.

In general case, if we denote by
$$\ZZ(\Frob_{\bar{v},A,\ell})_0 \subset \End_{\Z_{\ell}}(T_{\ell}(A))$$
the centralizer of $\Frob_{\bar{v},A,\ell}$ in $\End_{\Z_{\ell}}(T_{\ell}(A))$
then it follows from Tate's theorem on homomorphisms (Example \ref{finite}) that $\ZZ(\Frob_{\bar{v},A,\ell})_0$ is isomorphic as a $\Z_{\ell}$-algebra
to $\End(A(v))\otimes\Z_{\ell}$.

 By the Chebotarev density theorem, the set of all  $\Frob_{\bar{v},A,\ell}$'s (for all $v$) is everywhere dense in  $G_{\ell,A}$ \cite[Ch. I]{Serre}.

\end{sect}

Our main result is the following statement.

\begin{thm}
\label{main}
Let $A$ be an abelian variety of positive dimension over a number field $K$.
 Suppose that the groups $\GG_{\ell,A}$ are connected. Let $\P$ be a finite nonempty set of primes and for each $\ell \in P$ we are given an element
$$f_{\ell} \in \GG_{\ell,A}(\Q_{\ell})\subset  \Aut_{\Q_{\ell}}(V_{\ell}(A))$$
such that its characteristic polynomial 
$$P_{f_{\ell}}(t)=\det (t\II -f_{\ell}, V_{\ell}(A))\in \Q_{\ell}[t]$$
has no multiple roots. Let
 $$\ZZ(f_{\ell})_0\subset \End_{\Z_{\ell}}(T_{\ell}(A))$$
be the centralizer of $f_{\ell}$ in
$$\End_{\Z_{\ell}}(T_{\ell}(A)\subset \End_{\Q_{\ell}}(V_{\ell}(A).$$
Then the set of nonarchimedean places $v$ of $K$ such that 
the residual characteristic  $\fchar(k(v))$ does not belong to $\P$,  the abelian variety 
$A$ has good reduction $A(v))$ at $v$ and
$$\End(A(v))\otimes \Z_{\ell} \cong \ZZ(f_{\ell})_0 \ \forall \ell\in \P$$
has positive density.  (In addition, for all such $v$ the ring $\End(A(v))$ is commutative.)
\end{thm}

\begin{ex}
Let $A$ be an abelian variety of positive dimension over a number field $K$ and 
 suppose that the groups $\GG_{\ell,A}$ are connected.  Let $r$ be a positive integer and let  $\ell_1, \dots , \ell_r$ are $r$ {\sl  distinct primes}.
Suppose that for each $\ell_i$ we are given a  nonarchimedean place $\mathbf{v}_i$ of $K$ such that 
its residual characteristic  $\fchar(k(\mathbf{v}_i))\ne \ell_i$,  the abelian variety 
$A$ has good reduction $A(\mathbf{v}_i)$  at $\mathbf{v}_i$ and the endomorphism ring $\End(A(\mathbf{v}_i))$ is {\sl commutative}.
This implies that the characteristic polynomial of each Frobenius element $\Frob_{\bar{\mathbf{v}_i},A,\ell} \in G_{\ell,A}$ has {\sl no multiple roots}.
Recall that the centralizer $\ZZ(\Frob_{\bar{\mathbf{v}_i},A,\ell})_0$ is isomorphic as $\Z_{\ell}$-algebra to $\End(A(\mathbf{v}_i))\otimes \Z_{\ell}$.
Let us put $\P=\{\ell_1, \dots , \ell_r\}$.
It follows from Theorem \ref{main} that the set of nonarchimedean places $v$ of $K$ such that 
the residual characteristic  $\fchar(k(v))$ does not belong to $\P$,  the abelian variety 
$A$ has good reduction $A(v)$ at $v$ and
$$\End(A(v))\otimes \Z_{\ell_i} \cong \End(A(\mathbf{v}_i))\otimes \Z_{\ell_i} \ \forall i=1,  \dots , r.$$
has positive density.
\end{ex}

\begin{ex}
\label{elliptic}
Let $E$ be an elliptic curve without complex multiplication that is defined over a number field $K$. By a theorem of Serre \cite[Ch. IV, Sect. 2.2]{Serre},
$$\GG_{\ell,E}=\GL(V_{\ell}(E)).$$
In particular, $\GG_{\ell,E}$ is connected and isomorphic to the general linear group $\GL(2)$ over $\Q_{\ell}$ while 
$$\GG_{\ell,E}(\Q_{\ell})=\Aut_{\Q_{\ell}}(V_{\ell}(E)).$$
Let $\P$ be a finite nonempty set of primes.  For each $\ell \in \P$ let us fix a commutative {\sl semisimple} $2$-dimensional $\Q_{\ell}$-algebra $C_{\ell}$. Let us choose an {\sl order} $\Oc_{\ell}$ in $C_{\ell}$,  i.e., a $\Z_{\ell}$-subalgebra of $C_{\ell}$ (with the same $1$) that is a free $\Z_{\ell}$-submodule of rank $2$. Let us fix an isomorphism  of free $\Z_{\ell}$-modules
$$\Oc_{\ell} \cong T_{\ell}(E),$$
which extends by $\Q_{\ell}$-linearity to the isomorphism of $\Q_{\ell}$-vector spaces
$$C_{\ell}=\Oc_{\ell}\otimes_{\Z_{\ell}}\Q_{\ell} \cong 
T_{\ell}(E)\otimes_{\Z_{\ell}}\Q_{\ell}=V_{\ell}(E).$$
Multiplication in $C_{\ell}$ gives rise to an embedding
$$C_{\ell} \hookrightarrow \End_{\Q_{\ell}}(V_{\ell}(E));$$
further we will identify $C_{\ell}$ with its image in $\End_{\Q_{\ell}}(V_{\ell}(E))$. Clearly, $C_{\ell}$ coincides with its own centralizer in $\End_{\Q_{\ell}}(V_{\ell}(E))$. On the other hand, one may easily check (using the inclusion $1\in \Oc_{\ell}$) that
$$\Oc_{\ell}=\{u\in C_{\ell}|\mid u(T_{\ell}(E))\subset T_{\ell}(E)\}.$$
This implies that $\Oc_{\ell}$ coincides with the centralizer of $C_{\ell}$ in
$$\End_{\Z_{\ell}}(T_{\ell}(E))\subset \End_{\Q_{\ell}}(V_{\ell}(E)).$$
Since $C_{\ell}$ is $2$-dimensional, there exists $f_{\ell} \in C_{\ell}$ such that the pair $\{1, f_{\ell}\}$ is a basis of the $\Q_{\ell}$-vector space $C_{\ell}$.  Replacing $f_{\ell}$ by $1+\ell^M f_{\ell}$ for sufficiently big positive integer $M$, we may and will assume that
$$f_{\ell} \in C_{\ell}^{*}\subset \Aut_{\Q_{\ell}}(V_{\ell}(E)).$$
Clearly, the centralizer $\ZZ(f_{\ell})_0$ of $f_{\ell}$ in $\End_{\Z_{\ell}}(T_{\ell}(E))$ coincides with the centralizer of $C_{\ell}$ in  $\End_{\Z_{\ell}}(T_{\ell}(E))$. This implies that
$$\ZZ(f_{\ell})_0=\Oc_{\ell} \ \forall \ell \in \P.$$
Applying Theorem \ref{main}, we conclude that  the set of nonarchimedean places $v$ of $K$ such that 
the residual characteristic  $\fchar(k(v))$ does {\sl not} belong to $\P$, the elliptic curve 
$E$ has good reduction at $v$ and
$$\End(E(v))\otimes \Z_{\ell} \cong \Oc_{\ell} \ \forall \ell \in \P$$
has positive density.

 For example,  if $F$ is an imaginary quadratic field with the ring of integers $O_F$ and $N$ is a positive integer then let us consider the order
$\Lambda=\Z+N\cdot  O_F$ of conductor $N$ in $F$  and the collection of $\Z_{\ell}$-algebras
$$\Oc_{\ell}: =\Lambda \otimes \Z_{\ell}  \ \forall \ell \in \P.$$
We obtain that  the set $\Sigma(E,F,N)$ of all nonarchimedean places $v$ of $K$ such that the residual characteristic  $\fchar(k(v))$ does {\sl not} belong to $\P$, the elliptic curve $E$ has good ordinary reduction at $v$  and
$$\End(E(v))\otimes \Z_{\ell} \cong \Lambda \otimes \Z_{\ell}  \ \forall \ell \in \P$$
has positive density. In particular, this set is infinite.
\end{ex}

\begin{cor}
\label{discrE}
Let $E$ be an elliptic curve without CM that is defined over a number field $K$.
Let $N$ and $M$ be relatively prime positive integers. Let us consider the set $\tilde{\Sigma}(E,M,N)$   of nonarchimedean places $v$ of $K$ such 
$E$ has good ordinary reduction at $v$, the residual characteristic $\fchar(k(v))$ does not divide $NM$, the discriminant ${\bf \Delta}(v)$ of the order
$\End(E(v))$ is divisible by $N$ and the ratio  ${\bf \Delta}(v)/N$ is relatively prime to $MN$.  Then $\tilde{\Sigma}(E,M,N)$  contains a set of positive density.
In particular, $\tilde{\Sigma}(E,M,N)$   is infinite.
\end{cor}

\begin{rem}
Actually, one may prove that $\tilde{\Sigma}(E,M,N)$ has density, which is, of course, positive.
\end{rem}

\begin{rem}
The discriminant ${\bf \Delta}(v)$ is {\sl not} divisible by a prime $\ell$ if and only if either
$$\End(E(v))\otimes \Q_{\ell}=\Q_{\ell}\oplus \Q_{\ell}\supset \Z_{\ell}\oplus \Z_{\ell}=\End(E(v))\otimes \Z_{\ell}$$
or $\End(E(v))\otimes \Q_{\ell}$ is a {\sl field} that is an {\sl unramified} quadratic extension of $\Q_{\ell}$ and $\End(E(v))\otimes \Z_{\ell}$ is the
ring of integers in this quadratic field.
\end{rem}

\begin{proof}[Proof of Corollary \ref{discrE}]  Let $\P$ be the set of prime divisors of $MN$. Choose an imaginary quadratic field $F$,
whose discriminant is prime to $NM$ and put $\Lambda=\Z+N\cdot O_F$. Then $\tilde{\Sigma}(E,M,N)$ contains all the places of 
$\Sigma(E,F,N)$  except the finite set of places with residual characteristic dividing $M$.
 The set $\Sigma(E,F,N)$
 has positive density (see Example \ref{elliptic}), which would not change if we remove from it  finitely many places.
\end{proof}

\begin{rem}
Serre \cite[Ch. IV, Sect. 2.2, Exercises on pp. IV-13]{Serre} sketched a proof of the following assertion. 

 {\sl The set of nonarchimedean places $v$ of $K$ such that $\fchar(k(v))$ does not belong to $\P$, the elliptic curve
$E$ has good ordinary reduction at $v$ and
$$\End(E(v))\otimes \Q_{\ell} \cong C_{\ell} \ \forall \ell \in \P$$
has positive density.} 
In particular,  if one defines the set $\Sigma_{\P}(E)$  of all  places $v$ such that $E$ has good ordinary reduction  at $v$, the residual characteristic $\fchar(k(v))$ does not belong to $\P$ and the discriminant of the quadratic field $\End(E(v))\otimes\Q$ is divisible by all $\ell \in P$ then $\Sigma_{\P}(E)$ is infinite.
(See also \cite[Cor. 2.4 on p. 329]{Schoof}.)
\end{rem}

\begin{thm}
\label{jacobian}
Let $g \ge 2$ be an integer, $n=2g+1$ or $2g+2$.  Let $\P$ be a nonempty finite set of primes and suppose
that for each $\ell \in \P$  we a given a field $\mathcal{K}^{(\ell)}$ of characteristic different from $\ell$, a 
$g$-dimensional simple   abelian variety $B^{(\ell)}$   over $\mathcal{K}^{(\ell)}$
that admits a polarization of degree prime to $\ell$ and such that $\End^0(B^{(\ell)})$ is a number field of degree $2g$.
 (E.g., if $B$ is a principally polarized $g$-dimensional simple  complex abelian variety 
of CM type then we may take $B^{(\ell)}=B$ for all $\ell \in \P$.)

 Let $K$ be a number field and $f(x) \in K[x]$ be a degree $n$
 irreducible polynomial, whose Galois group over $K$ is either full symmetric group $\ST_n$ or the alternating group $\A_n$.
Let us consider the genus $g$ hyperelliptic curve $C_f:y^2=f(x)$ and its  jacobian $A$, which is a $g$-dimensional abelian variety over $K$.

Let $\Sigma$ be the set of all nonarchimedean places $v$ of $K$ such that $A$ has good reduction at $v$, the residual characteristic $\fchar(k(v))$ does not 
belong to $\P$ and the $\Z_{\ell}$-rings $\End(A)\otimes \Z_{\ell}$ and $\End((B^{\ell)})\otimes \Z{\ell}$ are isomorphic for all $\ell\in \P$.  Then $\Sigma$ has density $>0$.
\end{thm}

The paper is organized as follows. In Section \ref{symPol} we discuss $\ell$-adic symplectic groups that arise from polarizations on abelian varieties. Section \ref{NT} deals with trace forms and realated symplectic structures.
Section \ref{linearA} deals with centralizers of certain {\sl generic} elements of linear reductive groups over $\Q_{\ell}$. Section \ref{Felement} deals with applications of the Chebotarev density theorem for infinite Galois extensions of number fields  with $\ell$-adic Galois groups. In Section \ref{mainProof} we prove Theorems \ref{main} and \ref{jacobian}.

{\bf Acknowledgements}. I thank Mihran Papikian and Alina Cojocaru for stimulating questions and their interest in this paper.

\section{Polarizations and symplectic groups}
\label{symPol} 

Let $B$ be an abelian variety of positive dimension $g$ over a field $K$  and let $\ell$ be a prime that is different from $\fchar(K)$. We write 
$$\chi_{\ell}: \Gal(K) \to \Z_{\ell}^{*}$$
the {\sl cyclotomic character} that defines the Galois action on all $\ell$-power roots of unity. Let $\lambda$ be a polarization on $B$. Then $\lambda$ gives rise
to the altermating nondegenerate $\Z_{\ell}$-bilinear form
$$e_{\lambda,\ell}: T_{\ell}(B) \times T_{\ell}(B)  \to \Z_{\ell}$$
such that
$$e_{\lambda,\ell}(\rho_{\ell,B}(\sigma)x, \rho_{\ell,B}(\sigma)y)=\chi_{\ell}(\sigma)e_{\lambda,\ell}(x,y)$$
for all $\sigma\in \Gal(K)$ and $x,y \in T_{\ell}(B)$; in addition, $e_{\lambda,\ell}$ is perfect/unimodular if and only if $\deg(\lambda)$ is {\sl not} divisible by $\ell$
(see \cite{LangAV}).
Let us consider the (compact) group of symplectic similitudes
$$\Gp(T_{\ell}(B),e_{\lambda,\ell})=\{u \in \Aut_{\Z_{\ell}}(T_{\ell}(B))\mid \exists c \in \Z_{\ell}^{*} \text{ such that } e_{\lambda,\ell}(ux,uy)=
c \cdot e_{\lambda,\ell}(x,y)$$
 for all $x,y \in T_{\ell}(B)\}$.
Clearly,
$$G_{\ell,B}=\rho_{\ell,A}(\Gal(K))\subset \Gp(T_{\ell}(B),e_{\lambda,\ell})\subset \Aut_{\Z_{\ell}}(T_{\ell}(B))\subset \Aut_{\Q_{\ell}}(V_{\ell}(B)).$$
Extending $e_{\lambda,\ell}$ by $\Q_{\ell}$-linearity to $V_{\ell}(B)=T_{\ell}(B)\otimes_{\Z_{\ell}}\Q_{\ell}$, we obtain 
the altermating nondegenerate $\Z_{\ell}$-bilinear form
$$V_{\ell}(B) \times V_{\ell}(B)  \to \Q_{\ell},$$
which we continue to denote $e_{\lambda,\ell}$. Clearly,
$$e_{\lambda,\ell}(\rho_{\ell,B}(\sigma)x, \rho_{\ell,B}(\sigma)y)=\chi_{\ell}(\sigma)e_{\lambda,\ell}(x,y)$$
for all $\sigma\in \Gal(K)$ and $x,y \in V_{\ell}(B)$. Let us consider the  group of symplectic similitudes
$$\Gp(V_{\ell}(B),e_{\lambda,\ell})=\{u \in \Aut_{\Q_{\ell}}(V_{\ell}(B))\mid \exists c \in \Q_{\ell}^{*} \text{ such that } e_{\lambda,\ell}(ux,uy)=
c \cdot e_{\lambda,\ell}(x,y)$$
for all $x,y \in V_{\ell}(B)\}$.
Clearly, $\Gp(T_{\ell}(B),e_{\lambda,\ell})$ is an open compact subgroup of $\Gp(V_{\ell}(B),e_{\lambda,\ell})$ that coincides with the intersection
$\Gp(V_{\ell}(B),e_{\lambda,\ell})\bigcap \Aut_{\Z_{\ell}}(T_{\ell}(B))$. We have
$$G_{\ell,A} \subset \Gp(T_{\ell}(B),e_{\lambda,\ell})\subset \Gp(V_{\ell}(B),e_{\lambda,\ell})\subset \Aut_{\Q_{\ell}}(V_{\ell}(B)).$$
We write $\GSp(V_{\ell}(B), e_{\lambda,\ell})\subset \GL(V_{\ell}(B))$ for the connected linear reductive algebraic group of {\sl symplectic similitudes} over $\Q_{\ell}$ 
attached to $e_{\lambda,\ell}$.
Its group of $\Q_{\ell}$-points
$$\GSp(V_{\ell}(B), e_{\lambda,\ell})(\Q_{\ell})=\Gp(V_{\ell}(B),e_{\lambda,\ell})\subset \Aut_{\Q_{\ell}}(V_{\ell}(B))=\GL(V_{\ell}(B))(\Q_{\ell}).$$
Let us consider the finite-dimensional semisimple $\Q$-algebra $$\End^0(B)=\End(B)\otimes \Q.$$ We have the natural isomorpisms of $\Q$-algebras
$$[\End(B)\otimes\Z_{\ell}]\otimes_{\Z_{\ell}}\Q_{\ell}= \End(B)\otimes\Q_{\ell}=[\End(B)\otimes\Q]\otimes_{\Q}\Q_{\ell}=\End^0(B)\otimes_{\Q}\Q_{\ell}.$$
By (**), there is the natural embedding
$$\End^0(B)\otimes_{\Q}\Q_{\ell}=\End(B)\otimes\Q_{\ell} \hookrightarrow \End_{\Q_{\ell}}(V_{\ell}(B)).$$
We may view $\End^0(B)$ as the certain $\Q$-subalgebra of $\End^0(B)\otimes_{\Q}\Q_{\ell}$, identify the latter with its image in 
$\End_{\Q_{\ell}}(V_{\ell}(B))$ and get
$$\End^0(B)\subset \End^0(B)\otimes_{\Q}\Q_{\ell}\subset \End_{\Q_{\ell}}(V_{\ell}(B)).$$
The polarization $\lambda$ gives rise to {\sl Rosati involution} \cite{LangAV,Mumford} 
$$\End^0(B) \to \End^0(B), u \mapsto u^{\prime}$$
such that
$$e_{\lambda,\ell}(ux,y)=e_{\lambda,\ell}(x,u^{\prime}y) \ \forall x,y \in V_{\ell}(B).$$
This involution extends by $\Q_{\ell}$-linearity to the involution of the semisimple finite-dimensional $\Q_{\ell}$-algebra  $\End^0(B)\otimes_{\Q}\Q_{\ell}$
$$ \End^0(A)\otimes_{\Q}\Q_{\ell} \to  \End^0(B)\otimes_{\Q}\Q_{\ell}, u \mapsto u^{\prime}$$
such that
$$e_{\lambda,\ell}(ux,y)=e_{\lambda,\ell}(x,u^{\prime}y) \ \forall x,y \in V_{\ell}(B).$$
This implies that 
$$u \in [\End^0(B)\otimes_{\Q}\Q_{\ell}]^{*} \subset \Aut_{\Q_{\ell}}(V_{\ell}(B))$$
lies in $\Gp(V_{\ell}(B), e_{\lambda,\ell})$ if and only if
$$u^{\prime} u \in \Q_{\ell}^{*}\II.$$

The following statement will be used in the proof of Theorem \ref{jacobian}.

\begin{thm}
\label{CM}
Suppose that $\End^0(B)$ is a number field of degree $2g$. Then there exists an element $u \in \End(A)$  and a positive integer $q \in \Z$ such that
$$\End^0(A)=\Q[u], \
u^{\prime} u =q, \ u \in \Gp(V_{\ell}(A), e_{\lambda,\ell})$$
and the characteristic polynomial
$P_u(t)=\det(t\II-u, V_{\ell}(B))$ of $u$ has no multiple roots.

In addition, the centralizer $\ZZ(u)_0$ of $u$ in $\End_{\Z_{\ell}}(T_{\ell}(B)) \subset \End_{\Q_{\ell}}(V_{\ell}(B))$
coincides with $\End(A)\otimes\Z_{\ell}$. 
\end{thm}

In the course of the proof of Theorem \ref{CM} we will use the following statement that will be proven at the end of this section. (See also
\cite[Sect. 4]{ZarhinMMO}.)

\begin{lem}
\label{quadratic}
Let $Q$ be a field of characteristic zero, $F_0/Q$ a finite algebraic field extension and $F/F_0$ a quadratic field extension.
Let $\tau \in \Gal(F/F_0)$ be the only nontrivial element (involution) of the Galois group of $F/F_0$. Then there exists
$u \in F$ such that $F=Q[u]$ and $u \cdot \tau u=1$.
\end{lem}

\begin{proof}[Proof of Theorem \ref{CM}]
It follows from Albert's classification \cite{Mumford} (see also \cite{Oort}) that
the field $F:=\End^0(B)$ is a CM field and the Rosati involution coincides with the complex conjugation
$z \mapsto \bar{z}$ on $F$ and $R:=\End(B)$ is an order in $F$. Recall that $F$ is a purely imaginary quadratic extension 
of its totally real  number subfield $F_0$ and the complex conjugation is the only nontrivial element of the Galois group of $F/F_0$.

We have
$$F_{\ell}:=F\otimes \Q_{\ell}=\End^0(B)\otimes_{\Q}\Q_{\ell}\subset \End_{\Q_{\ell}}(V_{\ell}(B)).$$
Clearly, all elements of the commutative semisimple $\Q_{\ell}$-algebra $F_{\ell}$ act as semisimple linear operators in $V_{\ell}(B)$. 
The $F_{\ell}$-module $V_{\ell}(A)$ is {\sl free of rank $1$} \cite[Sect. 4, Th. 5(1)]{SerreTate}. This implies that
$F_{\ell}$ coincides with its own centralizer $\End_{F_{\ell}}(V_{\ell}(A))$ in $\End_{\Q_{\ell}}(V_{\ell}(A))$.
On the other hand, the intersection $F_{\ell}\bigcap \End_{\Z_{\ell}}(T_{\ell}(A))$ coincides
with
$$R_{\ell}:=R\otimes \Z_{\ell}=\End(B)\otimes\Z_{\ell}$$
 \cite[Sect. 4, Th. 5(1)]{SerreTate}. 

Suppose  that we have constructed an element $u \in R= \End(B)$ such that
$F=\End^0(B)=\Q[u]$ and $u^{\prime} u =q$ for some positive integer $q$.
This implies that the centralizer $\ZZ(u)$ of $u$ in $\End_{\Q_{\ell}}(V_{\ell}(B))$ coincides with the 
centralizer $\End_{F_{\ell}}(V_{\ell}(A))$ of $F_{\ell}$, i.e., equals $F_{\ell}$.  It follows that
 the centralizer $\ZZ(u)_0$ of $u$ in $ \End_{\Z_{\ell}}(T_{\ell}(B))$ coincides with the the intersection $F_{\ell}\bigcap \End_{\Z_{\ell}}(T_{\ell}(B))$,
i.e., equals $R_{\ell}$. In addition,  since $F_{\ell}$ is the centralizer of $u$ in $\End_{\Q_{\ell}}(V_{\ell}(B))$ and
$$\dim_{\Q_{\ell}}(F_{\ell})=2g=\dim_{\Q_{\ell}}(V_{\ell}(B)),$$
the characteristic polynomial $P_u(t)$ of $u$ has {\sl no} multiple roots. We have
$$e_{\lambda,\ell}(ux,uy)=e_{\lambda,\ell}(x,u^{\prime}uy)=   e_{\lambda,\ell}(x,q\cdot y) =q \cdot e_{\lambda,\ell}(x,y)$$
 for all $x,y \in V_{\ell}(B)$. This implies that
$$u \in  \Gp(V_{\ell}(A), e_{\lambda,\ell}).$$
Now let us construct such an $u$.  Applying Lemma \ref{quadratic} (to $Q=\Q$), we obtain the existence of $u_1 \in F$
with $\Q[u_1]=F$ and $u_1^{\prime}\cdot u_1=1$. Then there is a positive integer $m$ such that $u:=mu_1$ lies in $R$. Clearly,
$$\Q[b]=\Q[u_1]=F,  \ u^{\prime}=m u_1^{\prime}, \ u^{\prime} \cdot u =m^2 u_1^{\prime}\cdot u_1= m^2\cdot 1=m^2.$$
Now one has only to put $q=m^2$.
\end{proof}

\begin{proof}[Proof of Lemma \ref{quadratic}]
Recall that for each $u \in F$ the $Q$-subalgebra $Q[u]$ of $F$ generated by $u$ is actually a subfield, i.e.,
coincides with the (sub)field $Q(u)$.

Since $F/F_0$ is quadratic, $F=F_0(\sqrt{\delta})$ for some {\sl nonzero} $\delta\in F_0$.  We have
$$F=F_0+ F_0\cdot \sqrt{\delta}, \ \tau(\sqrt{\delta})=-\sqrt{\delta}$$
and $F_0$ coincides with the subfield of $\tau$-invariants in $F$.

Suppose that there is a nonzero  $\beta_0 \in F_0$
such that 
$F_0=Q(\delta \beta_0^2)$. 
%(For example, if $F_0=Q then we may take $\beta_0 =1$,)
Replacing if necessary, $\beta_0$ by $2\beta_0$, we may and will assume
that
$$\delta\beta_0^2+1\ne 0.$$
Let us put
$$\beta=\frac{\delta\beta_0^2-1}{\delta\beta_0^2+1}+\frac{2\beta_0}{\delta\beta_0^2+1}\cdot \sqrt{\delta}.$$
Clearly,  
 $$\beta \not\in F_0, \ \tau(\beta) =\frac{\delta\beta_0^2-1}{\delta\beta_0^2+1}-\frac{2\beta_0}{\delta\beta_0^2+1}\cdot \sqrt{\delta}, \ \tau(\beta)\cdot \beta=1$$
 and therefore $Q(\beta)$ contains $\tau(\beta)=1/\beta$, which implies that it contains both $\frac{\delta\beta_0^2-1}{\delta\beta_0^2+1}$ and 
$\frac{2\beta_0}{\delta\beta_0^2+1}\cdot \sqrt{\delta}$. This implies that $Q(\beta)$ contains $\delta\beta_0^2$ and therefore contains $Q(\delta\beta_0^2)$.
Since $Q(\delta\beta_0^2)=F_0$,  the subfield  $Q(\beta)$ contains $F_0$ and we have
$$F_0 \subset  Q(\beta)\subset F.$$
Since $F_0$ does not contain $\beta$,   $F_0 \ne Q(\beta)$ and therefore $Q(\beta)=F$.  We have
$$Q[\beta]=Q(\beta)=F.$$
This ends the proof if we find $\beta_0 \in F_0$ with
$$F_0=Q(\delta \beta_0^2).$$

Now let us construct such a $\beta_0$. If $F_0 =Q$ we may take any
$$\beta_0 \in Q=F, \ \beta_0 \ne 0, \ \beta_0^2 \ne -\frac{1}{\delta}.$$
Now suppose that $F_0 \ne Q$. Since $F_0/Q$ is separable, there is $\gamma \in F_0$ with $F_0=Q(\gamma)$.  
Clearly, $\gamma \not \in \Q\subset Q$; in particular,  $\gamma \ne 0$.  Since  separable  $F_0/Q$ contains only finitely many field subextensions of $Q$, there are two {\sl distinct}
positive integers $i,j \in \Z\subset Q$ such that  the subfields
$Q\left(\delta(\gamma+i)^2\right)$ and $Q\left(\delta(\gamma+j)^2\right)$ do coincide.
(Notice that $i^2 \ne j^2$.)
This implies that 
$2\delta(j-i)\gamma$ lies in $Q(\delta(\gamma+i)^2)$, i.e.,
$$\delta \cdot \gamma \in Q\left(\delta(\gamma+i)^2\right)=Q\left(\delta(\gamma+j)^2\right).$$
This implies that both
$$\frac{ (\gamma+i)^2}{\gamma}=\frac{\delta \cdot (\gamma+i)^2}{\delta \cdot \gamma}  \text{ and } \frac{ (\gamma+j)^2}{\gamma}=
\frac{\delta \cdot (\gamma+j)^2}{\delta \cdot \gamma}$$
lie in $Q(\delta(\gamma+i)^2)$.
This implies that
$$\left(2i-2j\right)+ \frac{i^2-j^2}{\gamma}=\frac{ (\gamma+i)^2}{\gamma}- \frac{ (\gamma+j)^2}{\gamma} \in Q\left(\delta(\gamma+i)^2\right).$$
Since $i^2 \ne j^2$, we conclude that $1/\gamma$ lies in $Q\left(\delta(\gamma+i)^2\right)$ and therefore
$$Q(\gamma)=F_0 \supset Q\left(\delta(\gamma+i)^2\right)\supset Q(1/\gamma)=Q(\gamma)=F_0.$$
This implies that $Q\left(\delta(\gamma+i)^2\right)=F_0$ and we may put $\beta_0=\gamma+i$.
\end{proof}

We finish this section by the following elementary (and probably well-known) statement that will be used later in Example \ref{GL}.

\begin{lem}
\label{generator}
Let $Q$ be a field of characteristic zero and  $C$ a finite-dimensional commutative semisimple $Q$-algebra. Then there exists an invertible element $u$ of $C$ such that
$C=Q[u]$.
\end{lem}

\begin{proof}
It is well known that commutative semisimple $C$ splits into a finite direct sum
$$C=\oplus_{i=1}^r C_i$$
where each $C_i$  is an overfield of $Q$. It is also clear that $C_i/Q$ is a finite algebraic field extension. Since we live in characteristic zero,
each $C_i/Q$ is separable and therefore there exists nonzero $z_i \in C_i$ such that $C_i=Q[z_i]$.  Let $\mathcal{P}_i(t)\in Q[t]$ be the minimal polynomial of $z_i$ over $Q$. By definition, $\mathcal{P}_i(t)$ is an irreducible monic polynomial of degree $[C_i:Q]$.  We have
$$\Z\subset \Q\subset Q\subset C_i.$$
We may choose integers $n_i \in \Z$ in such a way that all 
$\mathcal{P}_i(t+n_i)$ are distinct and do {\sl not} vanish at zero; in particular, they all are monic irreducible and therefore relatively prime to each other.  Clearly, $\mathcal{P}_i(t+n_i)$ is the minimal polynomial of $z_i-n_i$ over $K$. Clearly, $Q_j=Q[z_i]=Q[z_i-n_i]$.
This implies that the field $C_i$ is isomorphic as $Q$-algebra to the quotient $Q[t]/\mathcal{P}_i(t+n_i)Q[t]$. This  implies that the $Q$-algebra
$Q[t]/\{\prod_{i=1}^r \mathcal{P}_i(t+n_i)\}Q[t]$ is isomorphic to the direct sum $\oplus_{i=1}^r C_i=C$. Now one may take as $u$ the image of $t$ in $C$.
\end{proof}

\section{Trace forms}
\label{NT}
\begin{sect}
\label{NTgeneral}
Let $\ell$ be a prime.
Let $F_0/\Q_{\ell}$ be a field extension of finite degree $g$.
% and $F/F_0$ a quadratic field extension.
 Let $\Oc_0=\Oc_{0,\ell}$ be the ring of integers of the $\ell$-adic field  $F_{0,\ell}:=F_0$, which carries the natural structure of a free $\Z_{\ell}$-module of rank $g$. Let us fix an {uniformizer} $\pi \in \Oc_0$ that generates the maximal ideal in $\Oc_0$. We write
$$\Tr_0:=\Tr_{F_0/\Q_{\ell}}: F_0 \to \Q_{\ell}$$
for the ($\Q_{\ell}$-linear) trace map from $F_0$ to $\Q_{\ell}$. It is well known that the symmetric $\Q_{\ell}$-bilinear {\sl trace form}
$$B_{\Tr}: F_0 \times F_0 \to \Q_{\ell}, \ x,y \mapsto \Tr_0(xy)$$
is nondegenerate. This means  the homomorphism of $\Q_{\ell}$-vector spaces
$$\phi_{\Tr}: F_0 \to \Hom_{\Q_{\ell}}(F_0,\Q_{\ell})$$
that assigns to each $a \in F_0$ the $\Q_{\ell}$-linear map
$$B_{\Tr}(a,?): F_0 \to \Q_{\ell}, \ x \mapsto B_{\Tr}(a,x)=\Tr_0(ax)$$
is an isomorphism. Recall that the natural homorphism of $\Z_{\ell}$-algebras
$$\Oc_0\otimes_{\Z_{\ell}}\Q_{\ell} \to F_0,  \ x\otimes c\mapsto c\cdot x$$
is an isomorphism. This implies that the restriction map 
$$\Hom_{\Q_{\ell}}(F_0,\Q_{\ell}) \to\Hom_{\Z_{\ell}}(\Oc_0,\Q_{\ell})$$
is an isomorphism of $\Z_{\ell}$-modules. (Further we will identify these modules, using this isomorphism.) We have
$$\Hom_{\Z_{\ell}}(\Oc_0,\Z_{\ell})\subset \Hom_{\Z_{\ell}}(\Oc_0,\Q_{\ell})=\Hom_{\Q_{\ell}}(F_0,\Q_{\ell}).$$
The preimage
$$\mathcal{D}^{-1}:=\phi_{\Tr}^{-1}(\Hom_{\Z_{\ell}}(\Oc_0,\Z_{\ell}))\subset F_0$$
is the {\sl inverse different}, which is a fractional ideal in $F_0$ that contains $\Oc_0$ \cite[Ch. III, Sect. 3]{SerreCL}.
Since the obvious $\Z_{\ell}$-bilinear pairing of free $\Z_{\ell}$-modules of rank $g$
$$\Oc_0\times \Hom_{\Z_{\ell}}(\Oc_0,\Z_{\ell}) \to \Z_{\ell}$$
is unimodular, the  $\Z_{\ell}$-bilinear pairing of free $\Z_{\ell}$-modules of rank $g$
$$\Oc_0 \times \mathcal{D}^{-1} \to \Z_{\ell}, \ (x,y) \mapsto \Tr_0(xy)$$
is also unimodular.  Notice that there is a nonnegative integer $d$ such that
$$ \mathcal{D}^{-1}=\pi^{-d}\Oc_0\subset  F_0.$$
This implies that the symmetric  $\Z_{\ell}$-bilinear pairing 
$$\tilde{B}_{\Tr}:\Oc_0  \times \Oc_0  \to \Z_{\ell}, \ x,y \mapsto \Tr_0(\pi^{-d}xy)$$
is unimodular.
\end{sect}

Let $T=T_{\ell}$ be a free $\Oc_0$-module of rank $2$ provided with an alternating $\Oc_0$-bilinear unimodular form
$$e_0: T \times T \to \Oc_0.$$
Since $T$ has rank $2$, such a form exists and unique, up to multiplication by an element of $\Oc_0^{*}$. This implies that
if $u$ is an automorphism of $T$ then
$$e_0(ux,uy)=\det(u)\cdot e_0(x,y) \ \forall x,y \in T.$$
Let us consider the $2$-dimensional $\Oc_0\otimes_{\Z_{\ell}}\Q_{\ell}=F_0$-vector space
$$V=V_{\ell}:=T\otimes_{\Z_{\ell}}\Q_{\ell}$$
and extend $e_0$ by $F_0$-linearity to the alternating nondegenerate $F_0$-bilinear form
$$V \times V \to F_0,$$ 
which we continue to denote $e_0$.
Clearly, if $u \in \Aut_{F_0}(V)$ then
$$e_0(ux,uy)=\det(u)\cdot e_0(ux,uy) \ \forall x,y \in V.$$
Here and above
$$\det: \Aut_{F_0}(V)\cong \GL(2,F_0) \to F_0^{*}$$
is the determinant homomorphism.

\begin{lem}
\label{unimod}
The alternating $\Z_{\ell}$-bilinear form
$$e=e_{\ell}: T \times T \to \Z_{\ell},  \ x,y \mapsto  \Tr_0(\pi^{-d} e_0(x,y))$$
is unimodular.
\end{lem}

\begin{proof}
Let $l: T \to \Z_{\ell}$ be a $\Z_{\ell}$-linear map. We need to prove that there is exactly one $z \in T$ such that
$$l(x)=\Tr_0(\pi^{-d} e_0(x,z)) \  \forall x \in T.$$
In order to do that, let us choose a basis $\{f_1,f_2\}$ of the free $\Oc_0$-module $T$. Then $l$ gives rise (and is uniquely determined by) two $\Z_{\ell}$-linear maps
$$l_i: R \to \Z_{\ell}, \  a \mapsto l(a\cdot e_i)$$
for $i=1,2$.
We have
$$l(a_1 f_1+a_2 f_2)= l_1(a_1)+ l_2(a_2) \ \forall a_1,a_2 \in \Oc_0.$$
Since $\tilde{B}_{\Tr}$ is unimodular, there exists exacltly one $c_i \in \Oc_0$ with
$$l_i(a)=\tilde{B}_{\Tr}(a,c_i) \ \forall a \in \Oc$$
for $i=1,2$.
This implies that
$$l(a_1 f_1+a_2 f_2)=\tilde{B}_{\Tr}(a_1,c_1)+\tilde{B}_{\Tr}(a_2,c_2)=\Tr_0(\pi^{-d}\cdot a_1 c_1)+\Tr_0(\pi^{-d}\cdot a_2 c_2)=$$
$$\Tr_0(\pi^{-d}\cdot [a_1c_1+a_2 c_2]).$$
Since $e_0$ is unimodular, there is exactly one $z\in T$ with
$$e_0(f_1,z)=c_1, \ e_0(f_2,z)=c_2.$$
This implies that
$$e_0(a_1 f_1+a_2 f_2,z)=a_1 c_1 +a_2 c_2$$
and therefore
$$l(a_1 f_1+a_2 f_2,z)=\Tr_0[\pi^{-d}\cdot e_0(a_1 f_1+a_2 f_2,z)]=e(a_1 f_1+a_2 f_2, z).$$
\end{proof}

\begin{sect}
\label{NTalgebra}
Let $C=C_{\ell}$ be a $2$-dimensional commutative semisimple $F_0$-algebra. Then $C$ is either a quadratic field extension $F$  of $F_0$ or is isomorphic (as $F_0$-algebra) to the
direct sum $F_0\oplus F_0$.
Suppose that  $R=R_{\ell} \subset C$ is an  $\Oc_0$-subalgebra of $C$  that is a free $\Oc_0$-module of rank $2$. Clearly, the natural homomorphism of $\Oc_0$-algebras
$$R\otimes_{\Oc_0}F \to C,  \ x\otimes a \mapsto ax$$
is an isomorphism.

Suppose that $C=F$ is field (that is a quadratic extension of $F_0$). Then $\Oc_0 \subset R \subset \Oc$ where $\Oc$ is the ring of integers in the $\ell$-adic field $F$. This implies that there is a nonnegative integer $i$ such that
$$R=R_i:=\Oc_0+\pi^i \Oc\subset \Oc=R_0.$$
Conversely, for any nonnegative integer $i$ the $\Oc_0$-(sub)algebra $\Oc_0+\pi^i \Oc\subset C$ is a free $\Oc_0$-module of rank $2$.

Suppose that $C:=F_0\oplus F_0$ and let us put 
$$\Oc:=\Oc_0\oplus \Oc_0 \subset F_0\oplus F_0=C.$$
 We view $\Oc_0$ as the $\Oc_0$-subalgebra of $\Oc$ via the diagonal embedding. Then 
$$\Oc_0 \subset R \subset \Oc.$$
 This implies that there is a nonnegative integer such that
$$R=R_i:=\Oc_0+\pi^i \Oc\subset \Oc=R_0.$$
Conversely, for any nonnegative integer $i$ the $\Oc$-(sub)algebra $\Oc_0+\pi^i \Oc\subset C$ is a free $\Oc_0$-module of rank $2$.

%\begin{exs}
%\label{exC}
%\item
%Let $F/F_0$ be a quadratic field extension and $\Oc$ the ring of integers in $F$.  Let us put $C=F$.  Then for any nonnegative integer $i$ one may take
%$$R=R_i:=\Oc_0+\pi^i \Oc\subset \Oc=R_0.$$
%\item
%Let us consider $C:=F_0\oplus F_0$ and let us put 
%$$\Oc:=\Oc_0\oplus \Oc_0 \subset F_0\oplus F_0=C.$$
% We view $\Oc_0$ as the $\Oc_0$-subalgebra of $\Oc$ via the diagonal embedding. Then for any nonnegative integer $i$ one may take
%$$R=R_i:=\Oc_0+\pi^i \Oc\subset \Oc=R_0.$$
%\end{exs}

Let us fix an isomorphism $T \cong R$ of free $\Oc_0$-modules of rank $2$. This provides $T$ with the natural structure of the free $R$-module of rank $1$ and gives
rise to the embedding of $R$-algebras
$$R \hookrightarrow \End_{\Oc_0}(T),$$
which extends by  $F_0$-linearity the embedding of $F_0$-algebras
$$C=R\otimes_{\Oc_0}F_0  \hookrightarrow \End_{\Oc_0}(T)\otimes_{\Oc_0}F=\End_{F_0}(T\otimes_{\Oc_0}F)=\End_{F_0}(V).$$
Further we will identify $C$ with its image in $\End_{F_0}(V)$.
Clearly, $V$ becomes a free $C$-module of rank $1$. In particular, the centralizer of $C$ in $\End_{F_0}(V)$ coincides with $C$.  Since $T$ is a free $R$-module of rank $1$,
the centralizer of $C$ in $\End_{\Oc_0}(T)\subset \End_{F_0}(V)$ coincides with
$$R\subset C\subset \End_{F_0}(V).$$
Actually, we can do better and consider $V$ as the $2g$-dimensional $\Q_{\ell}$-vector space and $T$ as the $\Z_{\ell}$-lattice of rank $2g$ in $V$. Indeed,
$C$ is a finite-dimensional semisimple $\Q_{\ell}$-algebra of $\Q_{\ell}$-dimension $2g$ that acts faithfully on the $2g$-dimensional $\Q_{\ell}$-vector space $V$. This implies that
the centralizer
of $C$ even in $\End_{\Q_{\ell}}(V)$ coincides with $C$ and the centralizer of $C$ in $\End_{\Z_{\ell}}(T)\subset \End_{\Q_{\ell}}(V)$ coincides with
$$R\subset C\subset \End_{\Q_{\ell}}(V)$$
(recall that $T$ is a free $R$-module of rank $1$ and therefore the centralizer of $R$ in $\End_{\Z_{\ell}}(T)$ coincides with $R$).
If 
$$u \in C^{*}\subset \Aut_{F_0}(V)$$
and $c=\det(u) \in F_0^{*}$ actually lies in $\Q_{\ell}^{*}$ then
$$e(ux,uy)=\Tr_0(\pi^{-d} e_0(ux,uy))=\Tr_0(\pi^{-d} c \cdot e_0(x,y))=c \cdot \Tr_0(\pi^{-d}  \cdot e_0(x,y))=c \cdot e(x,y)$$
for all $x,y \in V$. This implies that $u$ lies in the group $\Gp(V,e)$ of symplectic similitudes.
\end{sect}

\begin{lem}
\label{generator2}
There exists $$\ub=\ub_{\ell} \in C^{*}= C_{\ell}^{*}$$ that lies in  $\Gp(V,e)=Gp(V_{\ell}, e_{\ell})$  and such that 
$$\Q_{\ell}[\ub_{\ell}]=\Q_{\ell}[\ub]=C=C_{\ell}.$$ In particular, the
centralizer of $\ub_{\ell}$ in  $\End_{\Z_{\ell}}(T_{\ell})\subset \End_{\Q_{\ell}}(V_{\ell})$ coincides with
$$R=R_{\ell}\subset C\subset \End_{\Q_{\ell}}(V_{\ell}).$$
\end{lem}

\begin{proof}
Suppose that $C$ is the quadratic over field $F$ of $F_0$. Let $\tau$ be the only nontrivial element (involution of $\Gal(F/F_0)$.
Then $V$ becomes a one-dimensional vector space over $F$ but we view $V$ as a $2$-dimensional $F_0$-vector space and each $u \in F$ acts on $V$ as a $F_0$-linear
operator thyat is multiplication by $u$.
Then the determinant $\det(u)$ of this operator is the {\sl norm} of $u$  with respect to $F/F_0$, i.e,,
$$\det(u)=u \cdot \tau(u).$$
This implies that if $u \cdot \tau(u)=1$ then $u$ lies in the symplectic group $\Sp(V,e)\subset \Gp(V,e)$.  So, we need to find $\ub\in F^{*}$ with
$$\ub\cdot \tau \ub=1, \ \Q_{\ell}[\ub]=F.$$
But the existence of such $\ub$ is guaranteed by Lemma \ref{quadratic} and we are done.

Suppose that $C=F_0\oplus F_0$.  Let
$$u=(u_1,u_2) \in F_0^{*}\times F_0^{*}=C^{*}.$$
Clearly $\det(u)=u_1u_2 \in  F_0^{*}$. This implies that if  $u_2=u_1^{-1}$ then 
$$\det(u)=u_1 u_2=u_1 u_1^{-1}=1$$
 and $u$ lies in the symplectic group $\Sp(V,e)\subset \Gp(V,e)$. 
By Lemma \ref{generator}, there exists a nonzero $u_1 \in F_0$  with $\Q_{\ell}[u_1]=F_0$. Replacing $u_1$ by $\ell^N u_1$ for sufficiently large positive integer $N$,
we may and will assume that 
$$0 \ne u_1 \in \ell\Oc_0 \subset \pi \Oc_0$$
and therefore $u_1^{-1}\not\in \Oc_0$.
This implies that  the degree $g$  {\sl minimal polynomial} $P_1(t)\in \Q_{\ell}[t]$ of $u_1$ over $\Q_{\ell}$ has coefficients in $\Z_{\ell}$, which is not the case
for  the degree $g$ (monic) {\sl minimal polynomial} $P_2(t)\in \Q_{\ell}[t]$ of $u_2=u_1^{-1}$ over $\Q_{\ell}$. Since both $P_1$ and $P_2$ are monic irreducible over $\Q_{\ell}$, they are relatively prime. This implies that if we put $\ub=(u_1,u_1^{-1}) \in F_0\oplus F_0$ then the $\Q_{\ell}$-(sub)algebra $\Q_{\ell}[\ub]$ of $F_0\oplus F_0$ is isomorphic to the quotient $\Q_{\ell}[t]/P_1(t)P_2(t) \Q_{\ell}[t]$ and therefore has $\Q_{\ell}$-dimension
$$\deg(P_1 P_2)=g+g=2g=\dim_{\Q_{\ell}}(F_0\oplus F_0)$$
and therefore
$$\Q_{\ell}[\ub]=F_0\oplus F_0=C.$$
\end{proof}

\section{Linear Algebraic groups over $\Q_{\ell}$}
\label{linearA}
The content of this section was inspired by exercises in Serre's book \cite[Ch. IV, Sect. 2.2]{Serre}.

\begin{sect}
Let $V$ be a vector space of finite positive dimension $d$ over $\Q_{\ell}$. We write $\II$ for the identity automorphism of $V$.
Let $T$ be a $\Z_{\ell}$-lattice in $V$ of (maximal) rank $d$. 
For every $u \in \End_{\Q_{\ell}}(V)$ we write
$\ZZ(u)$ for its centralizer in $\End_{\Q_{\ell}}(V)$ and $\Q_{\ell}[u]$ for the $\Q_{\ell}$-subalgebra in $\End_{\Q_{\ell}}(V)$ generated by $u$. We have
$$\II, u \in \Q_{\ell}[u]\subset \ZZ(u)\subset  \End_{\Q_{\ell}}(V).$$
Let us consider the intersection 
$$\ZZ(u)_0:=\ZZ(u)\bigcap \End_{\Z_{\ell}}(T)\subset \End_{\Z_{\ell}}(T)\subset \End_{\Q_{\ell}}(V).$$
Clearly, $\ZZ(u)_0$ coincides with the centralizer of $u$ in $\End_{\Z_{\ell}}(T)\subset \End_{\Q_{\ell}}(V)$.
It is also clear that  $\ZZ(u)_0$ is a $\Z_{\ell}$-subalgebra (order) in $\ZZ(u)$ and the natural map
$$\ZZ(u)_0\otimes_{\Z_{\ell}}\Q_{\ell} \to \ZZ(u),  \ u\otimes c \mapsto cu$$
is an isomorphism of $\Q_{\ell}$-algebras. 

 If  $u \in \End_{\Q_{\ell}}(V)$ then we consider its characteristic polynomial
$$P_u(t) =\det(t\II-u, V) \in \Q_{\ell}[t]$$
and define $\Delta(u)\in \Q_{\ell}$ as the {\sl discriminant} of $P_u(t)$.
For each $g \in \Aut_{\Q_{\ell}}(V)$
$$P_{gug^{-1}}(t)=P_u(t), \ \Delta(gug^{-1})=\Delta(u).$$
 The polynomial $P_u(t)$ has {\sl no} multiple roots if and only if $\Delta(u)\ne 0$. If this is the case
then $u$ is a semisimple (diagonalizable over $\bar{\Q}_{\ell}$) linear operator in $V$, the subalgebra 
$$\Q_{\ell}[u]\subset \End_{\Q_{\ell}}(V)$$
 is a commutive semisimple $\Q_{\ell}$-(sub)algebra of $\Q_{\ell}$-dimension $d$, which coincides with
$\ZZ(u)$.

Let 
 $\GG \subset \GL(V)$ be a connected reductive linear (sub)group of positive dimension.
Clearly 
$\GG(\Q_{\ell})$ is a closed subgroup of $\Aut_{\Q_{\ell}}(V)$ with respect to $\ell$-adic topology. One may view 
$$\Delta: u \mapsto \Delta(u)$$
 as a regular function on the affine algebraic variety $\GG$. Let us assume that
$\Delta$ is {\sl not} identically zero on $\GG$. 
%In other words,  there exists
%$$u \in \GG(\Q_{\ell})\subset  \Aut_{\Q_{\ell}}(V)$$
%such that $\Delta(u)\ne 0$,  $P_u(t)$ has no multiple roots and the $\Q_{\ell}$-(sub)algebra
%$\Q_{\ell}[u]$
% is a $d$-dimensional commutive semisimple $\Q_{\ell}$-(sub)algebra that coincides with
%$\ZZ(u)$. (Actually, there are plenty of such $u$.)

\end{sect}

%The following statement will be proven at the end of this section.

\begin{lem}
\label{DeltaZero}
%Suppose that the function  $\Delta$ is not identically zero on $\GG$.
 Let $G$ be an open compact subgroup in  $\GG(\Q_{\ell})$.  Then the subset
$$G_{\Delta}:=G \bigcap \{\Delta=0\}\subset G$$
has measure zero with respect to the Haar measure on $G$.
\end{lem}

\begin{proof}[Proof of Lemma \ref{DeltaZero}]
The group $G$ carries the natural structure of open compact $\ell$-adic subgroup of $\GG(\Q_{\ell})$; in addition, if $N$ is
the dimension of $G$ then $N$ coincides with the dimension of $\GG$.  Clearly,  every nonempty open (with respect to $\ell$-adic topology)
 subset of $G$ is dense in $\GG$ with respect to Zariski topology.  This implies that the {\sl interior} of $G_{\Delta}$ with respect to $\ell$-adic topology is {\sl empty}.
Notice that $G_{\Delta}$ is a closed {\sl analytical subspace} of $G$ that is stable under conjugation.  
 It is known  \cite[Sect. 4.2]{SerreCH} that there is a positive integer $a$ such
that for each positive integer $n$ there is an open subgroup $G(n)$ in $G$ with index
$$(G:G(n))=a \ell^{nN}.$$
In addition, there is a positive integer $b$ such that  the image $C_n$ of  $G_{\Delta}$ in the finite group $G/G(n)$ consists of, at most $b\ell^{n(N-1)}$ elements
(\cite[Example at the end of Sect. 4.1 and formula (73) of Sect. 4.2]{SerreCH}).    Since the (normalized) Haar measure of each coset of the subgroup $G(n)$ 
in $G$ is $1/(G:G(n))$, we conclude
that the Haar measure of $G_{\Delta}$ does not exceed 
$$m(n)=\frac{b\ell^{n(N-1)}}{a \ell^{nN}}.$$
Since $m(n)$ tends to $0$ while $n$ tends to $\infty$,  the Haar measure of $G_{\Delta}$ is zero.
\end{proof}

\begin{sect}
\label{general}
Let $\ub$ be an element of $\GG(\Q_{\ell})$ with $\Delta(\ub)\ne 0$, i,e., 
$P_{\ub}(t)$
has {\sl no} multiple roots.  
Then $\ub$ is semisimple and regular in $\GL(V)$  and therefore is a semisimple regular element of $\GG$. Recall that
the subalgebra 
$$\Q_{\ell}[\ub]\subset \End_{\Q_{\ell}}(V)$$
 is a commutative semisimple  $d$-dimensional $\Q_{\ell}$-(sub)algebra  that coincides with
$\ZZ(\ub)$.

Let $\TT$ be the {\sl maximal torus} in $\GG$  that  contains (regular) $\ub$. Since such a $\TT$ is unique \cite[Ch. IV, Sect. 12.2]{Borel}, it is defined over $\Q_{\ell}$ and we have
$$\ub \in \TT(\Q_{\ell}) \subset \GG(\Q_{\ell}) \subset \Aut_{\Q_{\ell}}(V).$$
Let us consider the subset $\TT^{\prime}(\ub)\subset \TT(\Q_{\ell})$ that consists of all 
$$u \in \TT(\Q_{\ell}) \subset \GG(\Q_{\ell}) \subset \Aut_{\Q_{\ell}}(V)$$
such that $\Delta(u)\ne 0$.
% their characteristic polynomial
%$$P_{u}(t)=\det(t\II-u, V) \in \Q_{\ell}[t]$$
%has {\sl no} multiple roots.
 Clearly, $\TT^{\prime}(\ub)$ is open everywhere dense  in $\TT(\Q_{\ell})$ with respect to $\ell$-adic topology,
it contains $\ub$ and all its elements are semisimple regular in $\GG$ and commute with $\ub$. Then for each $u\in \TT^{\prime}(\ub)$
$$\ZZ(u)\subset \End_{\Q_{\ell}}(V)$$
 is also a commutative semisimple $\Q_{\ell}$-(sub)algebra of $\Q_{\ell}$-dimension $d$ that coincides with $\Q_{\ell}[u]$; it also
contains $\ub$ and therefore contains
$\Q_{\ell}[\ub]$, which is also $d$-dimensional. This implies that
$$\Q_{\ell}[u]=\ZZ(u)=\ZZ(\ub)=\Q_{\ell}[\ub]\subset \End_{\Q_{\ell}}(V)$$
and therefore
$$\ZZ(u)_0=\ZZ(\ub)_0\subset \End_{\Z_{\ell}}(T)$$
for all $u \in \TT^{\prime}(\ub)$.

 Let us consider the map of $\ell$-adic manifolds
$$\Psi_{\ub}: \GG(\Q_{\ell})\times \TT^{\prime}(\ub) \to \GG(\Q_{\ell}), \ (g,u) \mapsto gug^{-1}.$$
Clearly,  
$\Delta$  does {\sl not} vanish on  the image of $\Psi$.
% has {\sl no} multiple roots.

It is known (\cite[p.  469, Proof of Th. 2.1]{Harder}, see also \cite[Proof of Prop. 7.3]{LP2}) that the {\sl tangent map} to $\Psi_{\ub}$ is everywhere {\sl surjective} (recall that every $u \in  \TT^{\prime}(\ub)$ is regular in $\GG$). This implies that  $\Psi_{\ub}$ is an {\sl open map}, i.e., the image under $\Psi_{\ub}$ of any open subset of $\GG(\Q_{\ell})\times \TT^{\prime}(\ub) $ is open in  $\GG(\Q_{\ell})$. In particular, if 
$G$ is an open compact subgroup in  $\GG(\Q_{\ell})$ then 
 $\TT^{\prime}(\ub)_G=\TT^{\prime}(\ub)\bigcap G$ is a (nonempty) open subset in  $G$, whose closure   contains $\II$ and therefore
the image  $\Psi_{\ub}(\TT^{\prime}(\ub)_G\times G)$ is an open subset in $G$, whose closure contains $\II$.
Notice that
$$\ZZ(gug^{-1})=g \ZZ(u)g^{-1} \ \forall g \in G.$$
If, if addition, 
$$G\subset \Aut_{\Z_{\ell}}(T)\subset \Aut_{\Q_{\ell}}(V)$$
then
$$\ZZ(gug^{-1})_0=g\ZZ(u)_0 g^{-1} = g \ZZ(\ub)_0 g^{-1} \ \forall g \in G.$$
In particular, the $\Z_{\ell}$-algebras $\ZZ(gug^{-1}))_0$ and $\ZZ(\ub)_0$ are isomorphic.
In addition, if $\ub \in G$ then 
$$\ub \in \Psi_{\ub}(\TT^{\prime}(\ub)_G\times G).$$
\end{sect}

\begin{thm}
\label{open}  Let $G$ be an open compact subgroup in $\GG(\Q_{\ell})$ that lies in
$$\Aut_{\Z_{\ell}}(T)\subset \Aut_{\Q_{\ell}}(V).$$
Let $\ub$ be an element of $\GG(\Q_{\ell})$ such that its characteristic polynomial
$$P_{\ub}(t)=\det(t\II-\ub, V) \in \Q_{\ell}[t]$$
has no multiple roots. Let us consider the set $X(\ub,T,G)$ of all elements $u \in G$ such that the $\Z_{\ell}$-algebra
$\ZZ(u)_0$ is isomorphic to 
$\ZZ(\ub)_0$.
  Then $X(\ub,T,G)$ is a nonempty open subset in $G$ that is stable under conjugation. Its boundary lies in 
$G_{\Delta}$ and contains $\II$. 
\end{thm}

%\begin{rem}
%\label{VT}
%Clearly, $\ZZ(u)_0$ coincides with the centralizer of $u$ in $\End_{\Z_{\ell}}(T)\subset \End_{\Q_{\ell}}(V)$.
%\end{rem}

\begin{rem}
\label{disjoint}
Suppose that $\ub_1$ and $\ub_2$ are elements of $\GG(\Q_{\ell})$ with
$$\Delta(\ub_1)\ne 0, \ \Delta(\ub_2)\ne 0.$$
Let us consider elements
$$u_1 \in X(\ub_1,T,G), \ u_2 \in X(\ub_2,T,G).$$
We have isomorphisms of $\Z_{\ell}$-algebras
$$\ZZ(\ub_1)_0 \cong \ZZ(u_1)_0, \  \ZZ(\ub_2)_0 \cong \ZZ(u_2)_0.$$
This implies that either $\ZZ(\ub_1)_0$ and $\ZZ(\ub_2)_0$ are {\sl isomorphic} and
$$u_1 \in X(\ub_2,T,G), \ u_2 \in X(\ub_1,T,G)$$
or they are {\sl not} isomorphic and
$$u_1\not \in X(\ub_2,T,G), \ u_2 \not\in X(\ub_1,T,G).$$
It follows that the subsets $X(\ub_1,T,G),$ and $X(\ub_2,T,G)$ either {\sl coincide} or do {\sl not} meet each other.
\end{rem}

\begin{proof}
Clearly,   $X(\ub,T,G)$  is stable under conjugation,
$\ub$ lies in $X(\ub,T,G)$ while $\II$ does {\sl not} belong to $X(\ub,T,G)$.
In the notation above,  $\Psi_{\ub}(\TT^{\prime}(\ub)_G\times G)$ is an open subset in $G$, whose closure contains $\II$ (and therefore $\II$ lies on the boundary)
 and such that
for each $u \in  \Psi_{\ub}(\TT^{\prime}_G\times G)$ the $\Z_{\ell}$-algebra 
$\ZZ(u)_0$ is isomorphic to $\ZZ(\ub)_0$.  This implies that $X(\ub,T,G)$ contains open  $\Psi_{\ub}(\TT^{\prime}(\ub)_G\times G)\subset G$. In particular, 
$X(\ub,T,G)$ is nonempty and its closure in $G$ contains $\II$. It remains to prove that $X(\ub,T,G)$ is open.  Let  $u_1$ be an element of
$X(\ub,T,G)$.  Clearly,
$$X(\ub,T,G)=X(u_1,T,G).$$
On the other hand, the  centralizer $\ZZ(u_1)$ of $u_1$ in $\End_{\Q_{\ell}}(V)$ is isomorphic to $\ZZ(\ub)$, i.e., is a semisimple commutative $\Q_{\ell}$-algebra 
of $\Q_{\ell}$-dimension $d$ where $d=\dim_{\Q_{\ell}}(V)$.  This means that the characteristic polynomial of $u_1$ has no multiple roots and therefore
 (replacing $\ub$ by $u_1$) we may define $\TT^{\prime}(u_1)$, $\Psi_{u_1}$ and $\TT^{\prime}(u_1)_G$.   Since $u_1$ is an element of $G$,
 it lies in the open subset
 $\Psi_{u_1}(\TT^{\prime}(\ub)_G\times G)$ of $G$. On the other hand,
$$\Psi_{u_1}(\TT^{\prime}(\ub)_G\times G)\subset X(u_1,T,G)=X(\ub,T,G)$$
which proves the openness of $X(\ub,T,G)$. 

We still have to check that $\Delta$ vanishes identically on the boundary of $X(\ub,T,G)$.
In order to do that, recall (Remark \ref{disjoint}) that if 
$$u \in G, \ \Delta(u) \ne 0$$
then either $X(\ub,T,G)=X(u,T,G)$ or these two open  subsets of $G$ do {\sl not} meet each other. Taking into account that $u\in X(u,T,G)$, we obtain that
$$\{G \setminus G_{\Delta}\}\setminus X(\ub,T,G)$$
 coincides with the union of all (open) $X(u,T,G)$ where $u$ runs through the (same!) set $\{G \setminus G_{\Delta}\} \setminus X(\ub,T,G)$. This implies that 
$G  \setminus\{ G_{\Delta} \bigcup X(\ub,T,G)\}$ is an {\sl open} subset in $G$ that obviously does {\sl not} meet $X(\ub,T,G)$.
This implies that the closure of $X(\ub,T,G)$ lies in
$$X(\ub,T,G) \bigcup G_{\Delta}.$$
Since $X(\ub,T,G)$ is open, its boundary lies in $G_{\Delta}$. On the other hand, we have seen in Section \ref{general} that $\II$ lies in the closure
of $X(\ub,T,G)$ but not in $X(\ub,T,G)$. This implies that $\II$ lies in the boundary of
of $X(\ub,T,G)$ b
\end{proof}

\begin{ex}
\label{GL}
Suppose that $\GG=\GL(V)$.
Let $C$ be a $d$-dimensional semisimple commutative $\Q_{\ell}$-algebra and $R\subset C$ an order in $C$,
i.e., a $\Z_{\ell}$-subalgebra of $C$ (with the same $1$) that is a free $\Z_{\ell}$-module of rank $d$. By Lemma \ref{generator}, there exists
$\ub \in C^{*}$ such that $C=\Q_{\ell}[\ub]$.  
Let us fix an isomorphism of free $\Z_{\ell}$-modules
$R \cong T$
and use it in order to provide $T$ with the structure of a free $R$-module of rank $1$. 
Tensoring by $\Q_{\ell}$, we obtain the natural structure of $R\otimes_{\Z_{\ell}}\Q_{\ell}=C$-module
on $ T\otimes_{\Z_{\ell}}\Q_{\ell}=V$.
This gives us the $\Q_{\ell}$-algebra
embedding $C \hookrightarrow \End_{\Q_{\ell}}(V)$ in such a way that $R\subset C$ lands in $\End_{\Z_{\ell}}(T)\subset \End_{\Q_{\ell}}(V)$.
Further we will identify $C$ and $R$ with their images in  $\End_{\Q_{\ell}}(V)$ and  $\End_{\Z_{\ell}}(T)$ respectively. (In particular, we may view
$\ub$ as an element of $C^{*}\subset \Aut_{\Q_{\ell}}(V)$.)
Since $\ub$ lies in semisimple commutative $C\subset \End_{\Q_{\ell}}(V)$, it is a semisimple linear operator in $V$.
This provides $V$ with the natural structure of a free $C$-module of rank $1$; in particular, the centralizer $\End_C(V)$ of $C$ in 
$\End_{\Q_{\ell}}(V)$ coincides with  $C$.  Similarly, $T$ becomes a free $R$-module of rank $1$ and the centralizer $\End_R(T)$ of $R$ in 
$\End_{\Z_{\ell}}(T)$ coincides with  $R$. It follows that the centralizer of $\ub$ in $\End_{\Q_{\ell}}(V)$ coincides with $C$ and therefore
the centralizer $\ZZ(\ub)_0$ of $\ub$ in $\End_{\Z_{\ell}}(T)\subset \End_{\Q_{\ell}}(V)$ coincides with $R$. In particular, the $\Q_{\ell}$-dimension of the centralizer
of semisimple $\ub$ in  $\End_{\Q_{\ell}}(V)$ coincides with $\dim_{\Q_{\ell}}(V)$ and therefore the characteristic polynomial of $\ub$ has {\sl no} multiple roots.

 Let $\mathbf{X}(R,T,G)$ be the set of all $u \in G$
such that $\ZZ(u)_0$ is isomorphic as $\Z_{\ell}$-algebra to $R$. Then
$$\mathbf{X}(R,T,G)=X(\ub,T,G).$$
It follows from Theorem \ref{open} that $\mathbf{X}(R,T,G)$ is an open nonempty subset of $G$, whose closure contains the identity element and the boundary has
measure zero with respect to the Haar measure on $G$.
\end{ex}

\begin{ex}
\label{symplecticSplit}
Suppose $d=2g$ is even,   $\CC$ is a $g$-dimensional semisimple commutative $\Q_{\ell}$-algebra and $\RR\subset \CC$ is an order in $\CC$.
Let $\mathcal{T}$ be a a free $\RR$-module of rank $1$. Then $\mathcal{V}=\mathcal{T}\otimes_{\Z_{\ell}}\Q_{\ell}$ is a free 
$\RR\otimes_{\Z_{\ell}}\Q_{\ell}=\CC$-module of rank $1$. We may view $\mathcal{T}$ as a rank $g$ $\Z_{\ell}$-lattice (and a $\RR$-submodule) in $\mathcal{V}$.
Let us consider the free $\RR$-module $\mathbf{T}=\mathcal{T}\oplus \Hom_{\Z_{\ell}}(\mathcal{T},\Z_{\ell})$ of rank $2$, which is a  rank $2g$ $\Z_{\ell}$-lattice  in the
$2g$-dimensional vector space $\mathbf{V}=\mathcal{V}\oplus \Hom_{\Q_{\ell}}(\mathcal{V},\Q_{\ell})$. Notice that  $\mathbf{V}$ carries the natural structure of a free $\CC$-module of rank $2$ and we have a natural embedding
$$\CC \oplus \CC \hookrightarrow \End_{\Q_{\ell}}(\mathcal{V}) \oplus  \End_{\Q_{\ell}}(\Hom_{\Q_{\ell}}(\mathcal{V},\Q_{\ell}))\subset 
\End_{\Q_{\ell}}[\mathcal{V}\oplus \Hom_{\Q_{\ell}}(\mathcal{V},\Q_{\ell})]=\End_{\Q_{\ell}}(\mathbf{V})$$
such that each $(u_1,u_2) \in \CC \oplus \CC$ sends $(x, l)\in \mathcal{V}\oplus \Hom_{\Q_{\ell}}(\mathcal{V},\Q_{\ell})$ to $(u_1 x, l u_2)$.
Further we will identify $\CC \oplus \CC$ with its image in $\End_{\Q_{\ell}}(\mathbf{V})$. Under this identification the subring 
$\RR\oplus \RR\subset \CC \oplus \CC$ lands in
$$\End_{\Z_{\ell}}(\mathcal{T}) \oplus  \End_{\Z_{\ell}}(\Hom_{\Z_{\ell}}(\mathcal{T},\Z_{\ell}))\subset 
\End_{\Z_{\ell}}(\mathcal{T}\oplus \Hom_{\Z_{\ell}}(\mathcal{T},\Z_{\ell}))=\End_{\Z_{\ell}}(\mathbf{T}).$$
Clearly, $\CC \oplus \CC$ coincides with its own centralizer in $\End_{\Q_{\ell}}(\mathbf{V})$ and 
 $\RR \oplus \RR$ coincides with its own centralizer in $\End_{\Z_{\ell}}(\mathbf{T})$. Notice that the
$\Q_{\ell}$-dimensions of $\CC \oplus \CC$ and $\mathbf{V})$ do coincide.

There is a perfect alternating $\Z_{\ell}$-bilinear form
$$e: \mathbf{T} \times \mathbf{T} \to \Z_{\ell}, \  (x_1, l_1), (x_2,l_2) \mapsto l_1(x_2)-l_2(x_1)$$
for all
$$x_1,x_2 \in \mathcal{T}, \ l_1,l_2 \in \Hom_{\Z_{\ell}}(\mathcal{T},\Z_{\ell}).$$
This form extends by $\Q_{\ell}$-linearity to the nondegenerate alternating $\Q_{\ell}$-bilinear form
$$\mathbf{V} \times \mathbf{V} \to \Q_{\ell}, \  (x_1, l_1), (x_2,l_2) \mapsto l_1(x_2)-l_2(x_1)$$
$$\forall x_1,x_2 \in \mathcal{V}, \ l_1,l_2 \in \Hom_{\Z_{\ell}}(\mathcal{V},\Q_{\ell}),$$
which we also denote by $e$.

Let $\GG=\GSp(\mathbf{V},e)\subset \GL(\mathbf{V})$ be the (connected) reductive algebraic $\Q_{\ell}$-group of symplectic similitudes of $\mathbf{V}$ attached to $e$. We have
$$\GG(\Q_{\ell})=\GSp(\mathbf{V},e)(\Q_{\ell})=\Gp(\mathbf{V},e).$$
If $u_1 \in \CC^{*}$ and $q \in \Q_{\ell}^{*}$  then the element 
$(u_1, q u_1^{-1}) \in (\CC\oplus \CC)^{*}\subset \Aut_{\Q_{\ell}}(\mathbf{V})$ lies in $\Gp(\mathbf{V},e)$.
When $q=1$ this element lies in $\Sp(\mathbf{V},e)$.

Using Example \ref{GL}, choose  $u_1 \in \CC^{*}\subset \Aut_{\Q_{\ell}}(\mathcal{V})$  such that  the characteristic polynomial 
$P_{u_1}(t)$ has no multiple roots,
$\Q_{\ell}[u_1]=\CC$ and the
centralizer $\ZZ[u_1]_0$ of $u_1$ in $\End_{\Z_{\ell}}(\mathcal{T})\subset _{\Q_{\ell}(\mathcal[V})$ coincides with $\RR$.
We may choose $q$   in such a way that the characteristic polynomial 
$$P_{qu_1^{-1}}(t)= (t/q)^g P_{u_1}(q/t)$$
 of $qu_1^{-1}$ has no common roots with $P_u(t)$. (E.g., pick an integer $N$
such that none of the roots of $P_u(t)$ is of the form $\pm \ell^N$ and put   $q=\ell^{2N}$.) Then the characteristic polynomial of
$$\ub=(u_1, q u_1^{-1}) \in (\CC\oplus \CC)^{*}\subset \Aut_{\Q_{\ell}}(\mathbf{V})$$
coincides with the product $(t/q)^g P_{u_1}(q/t) \cdot P_{u_1}(t)$
and therefore
has no multiple roots. It follows that $\Q_{\ell}[\ub]=\CC\oplus\CC$ and therefore  the centralizer of 
$\ub$ in $\End_{\Q_{\ell}}(\mathbf{V})$ coincides with $\CC\oplus\CC$ and therefore
the centralizer $\ZZ(\ub)_0$ of $\ub$ in $\End_{\Z_{\ell}}(\mathbf{T})\subset \End_{\Q_{\ell}}(\mathbf{V})$ coincides with $\RR\oplus\RR$.

Let $G \subset \Aut_{\Z_{\ell}}(\mathbf{T})$ be an open compact subgroup in $\Gp(\mathbf{V},e)$.
 Let $\mathbf{X}(\RR\oplus\RR,\mathbf{T},G)$ be the set of all $u \in G$
such that $\ZZ(u)_0$ is isomorphic as $\Z_{\ell}$-algebra to $\RR\oplus\RR$. Then
$$\mathbf{X}(\RR\oplus\RR,\mathbf{}T,G)=X(\ub,\mathbf{T},G).$$
It follows from Theorem \ref{open} that $\mathbf{X}(\RR\oplus\RR,\mathbf{T},G)$ is an open nonempty subset of $G$, whose closure contains the identity element and the boundary has
measure zero with respect to the Haar measure on $G$.

\end{ex}

\begin{cor}
\label{manyELL}
Let $\mathcal{G}$ be a compact profinite topological group.
Let $\P$ be a nonempty  finite set of primes.

 Suppose that for each $\ell\in \P$ we are given the following data.
\begin{itemize}
\item
A $\Q_{\ell}$-vector space $V_{\ell}$ of finite positive dimension $d_{\ell}$ provided with a $\Z_{\ell}$-lattice $T_{\ell}\subset V_{\ell}$ of rank $d_{\ell}$.
\item
A connected reductiive linear algebraic subgroup $\GG_{\ell}\subset \GL(V_{\ell})$ of positive dimension.
\item
An element
$$\ub_{\ell} \in\GG_{\ell}(\Q_{\ell}) \subset \Aut_{\Q_{\ell}}(V_{\ell})$$
such that its characteristic polynomial
$$P_{\ub_{\ell}}(t)=\det(t\II-\ub_{\ell},V) \in \Q_{\ell}[t]$$
has no multiple roots. We write $\ZZ(\ub_{\ell})_0$ for the centralizer of $\ub_{\ell}$ in $\End_{\Z_{\ell}}(T_{\ell})\subset \End_{\Q_{\ell}}(V_{\ell})$.
\item
A continuous homomorphism of topological groups
$$\rho_{\ell}: \mathcal{G} \to \Aut_{\Z_{\ell}}(T_{\ell})\subset \Aut_{\Q_{\ell}}(V_{\ell}),$$
whose image 
$$G_{\ell}:= \rho_{\ell}(\mathcal{G}) \subset \Aut_{\Z_{\ell}}(T_{\ell})\subset \Aut_{\Q_{\ell}}(V_{\ell}),$$
  is an open subgroup in $\GG_{\ell}(\Q_{\ell})$.
\end{itemize}
Let us consider the subset $Y_{\ell}  \subset \mathcal{G}$ that consists of all $\sigma \in \mathcal{G}$ such that the centralizer $\ZZ(\rho_{\ell}(\sigma))_0$ of 
$\rho_{\ell}(\sigma)$ in $\End_{\Z_{\ell}}(T_{\ell})\subset \End_{\Q_{\ell}}(V_{\ell})$ is isomorphic (as a $\Z_{\ell}$-algebra) to $\ZZ(\ub_{\ell})_0$.  

Let us consider the product-homomorphism
$$\rho:=\prod_{\ell\in\P}\rho_{\ell}:\mathcal{G} \to \prod_{\ell\in\P} G_{\ell}, \ \sigma\mapsto \{\rho_{\ell}(\sigma)\}_{\ell\in \P}.$$

Then
 the image $\rho(\mathcal{G})$ is an open subgroup of finite index in $\prod_{\ell\in\P} G_{\ell}$ and
 the intersection 
$$Y:=\bigcap_{\ell\in\P}Y_{\ell}\subset \mathcal{G}$$ 
of all $Y_{\ell}$ is an open nonempty  subset in $ \mathcal{G}$ that is stable under conjugation and its closure contains the identity element of $\mathcal{G}$.
\end{cor}

\begin{proof}
Clearly,  (every $Y_{\ell}$ and therefore)  $Y$ is stable under conjugation.
 It follows from Theorem \ref{open} that every 
$$X(\ub_{\ell},T_{\ell},G_{\ell}) \subset G_{\ell}$$
 is an open nonempty  subset of $G_{\ell}$ and its closure contains the identity element of $G_{\ell}$. Clearly,
$$Y_{\ell}=\rho_{\ell}^{-1}(X(\ub_{\ell},T_{\ell},G_{\ell})) \subset \mathcal{G}.$$
This implies that every
 $Y_{\ell}$  is an open nonempty  subset in $\mathcal{G}$ and its closure contains the identity element of $\mathcal{G}$. This implies
that $Y$ is also open. It remains to check that $Y$ is nonempty and its closure contains the identity element.  In order to do that, notice that each $G_{\ell}$ contains
an open subgroup of finite index that is a pro-$\ell$-group. So, there is an open subgroup $\mathcal{G}_1$ of finite index in $\mathcal{G}$ such that
 $G_{\ell,1}:=\rho_{\ell}(\mathcal{G}_1)$ is a pro-$\ell$-group. Clearly,  $G_{\ell,1}$ is a closed subgroup of finite index in $G_{\ell}$ and therefore is open in $G_{\ell}$ and therefore is open in $\GG_{\ell}(\Q_{\ell})$ as well.

Let us consider the product-homomorphism
$$\rho_1:\mathcal{G}_1 \to \prod_{\ell\in\P} G_{\ell,1}, \ \sigma\mapsto \{\rho_{\ell}(\sigma)\}_{\ell\in \P}.$$
The image $\rho_1(\mathcal{G}_1 )\subset \prod_{\ell\in\P} G_{\ell,1}$ is a compact subgroup that maps surjectively on each factor $G_{\ell,1}$. Since 
the $G_{\ell,1}$'s are pro-$\ell$-groups for pair-wise $\ell$, we have 
$$\rho_1(\mathcal{G}_1)= \prod_{\ell\in\P} G_{\ell,1},$$
i.e., $\rho_1$ is {\sl surjective}.
(Compare with \cite[Proof of Prop. 7.1]{LP2}. Actually, this argument goes back to Serre \cite[Ch. IV,  Sect.  2.2,  Exercise 3c on p. IV-14]{Serre}.)
Since $\rho_1$ is {\sl surjective},
$$Y\bigcap \mathcal{G}_1=\rho_1^{-1}(\prod_{\ell\in\P}X(\ub_{\ell},T_{\ell},G_{\ell,1})) \subset \mathcal{G}_1$$
is nonempty (as the preimage of a nonempty subset) and its closure contains the identity element of $\mathcal{G}_1$.
\end{proof}

\begin{cor}
\label{measureY}
We keep the notation and assumptions of Corollary \ref{manyELL}. Assume additionally that $\mathcal{G}$ is a closed subgroup of $\prod_{\ell\in\P} G_{\ell}$ and $\rho_{\ell}:\mathcal{G}\to G_{\ell}$ coincides with the corresponding projection map (for all $\ell\in P$). Then $\mathcal{G}$ is an open subgroup of finite index in $\prod_{\ell\in\P} G_{\ell}$,  
$$Y=\mathcal{G}\bigcap \prod_{\ell\in\P}X(\ub_{\ell},T_{\ell},G_{\ell})\subset \mathcal{G}$$
is on open nonempty subset of $\mathcal{G}$
  while the boundary of $Y$ in $\mathcal{G}$ contains the identity element of $\mathcal{G}$ and has measure zero with respect to the Haar measure on $\mathcal{G}$.
\end{cor}

\begin{proof}
Clearly, $\mathcal{G}$ is compact.
It follows from Corollary \ref{manyELL} that $\mathcal{G}$ is an open subgroup of finite index in $\prod_{\ell\in\P} G_{\ell}$. By the definition of $Y$,
$$Y=\mathcal{G}\bigcap \prod_{\ell\in\P}X(\ub_{\ell},T_{\ell},G_{\ell})\subset \prod_{\ell\in\P} G_{\ell}.$$ It follows that the closure $\bar{Y}$ of $Y$ lies in
$$\prod_{\ell\in\P}\left[X(\ub_{\ell},T_{\ell},G_{\ell})\bigsqcup (G_{\ell})_{\Delta}\right]
\subset \prod_{\ell\in\P} G_{\ell}.$$
Recall (Corollary \ref{manyELL}) that $Y$ is open in $\mathcal{G}$.
This implies that the boundary $\partial Y$ of $Y$ lies in the (finite)  union $Z$ of products
$$Z_p:=(G_{p})_{\Delta}\times \prod_{\ell\in\P, \ell \ne p}G_{\ell}$$
for all $p \in P$. By Lemma \ref{DeltaZero}, $(G_{p})_{\Delta}$ has measure zero with respect to the Haar measure on $G_{p}$. This implies that each product-set $Z_p$ has measure zero with respect to the Haar measure on $\prod_{\ell\in\P} G_{\ell}$. It follows that their union $Z$ and therefore  its subset $\partial Y$ have measure zero with respect to the Haar measure on $\prod_{\ell\in\P} G_{\ell}$. Since $\partial Y$  lies in $\mathcal{G}$, which is an open subgroup of finite index in $\prod_{\ell\in\P} G_{\ell}$, the boundary $\partial Y$   has measure zero with  respect to the Haar measure on
$\mathcal{G}$ as well.
\end{proof}

\begin{rem}
\label{positive}
Since $Y$ is open nonempty in  $\mathcal{G}$, 
 its measure (with respect to the Haar measure) is positive.
\end{rem}

%\begin{rem}
%\label{index}
%Let $\mathcal{G}^{\prime}$ be an open subgroup of finite index in %$\mathcal{G}$. Then $\mathcal{G}^{\prime}$ is closed in 
%$\mathcal{G}$.
%Let us put
%$$Y^{\prime}:=Y\bigcap \mathcal{G}^{\prime}\subset \mathcal{G}^{\prime}.%$$
%Then $Y^{\prime}$ is an open subset of $\mathcal{G}^{\prime}$, with closure %containing the identity element; in particular,
%$Y^{\prime}$ is nonempty.  Clearly, the boundary $\partial Y^{\prime}$ of  %$\mathcal{G}^{\prime}$ coincides with the intersection
%${\partial Y}\bigcap \mathcal{G}^{\prime}\subset {\partial Y}$ and therefore %also has measure zero with respect to the Haar measure on
%$\mathcal{G}$. Since $\partial Y^{\prime}$  lies in $\mathcal{G}^{\prime}$ %and the latter is an open subgroup of finite index in $\mathcal{G}$, the boundary %$\partial Y$   has measure zero with  respect to the Haar measure on
%$\mathcal{G}^{\prime}$ as well. On the other hand, since $Y^{\prime}$ is an %open nonempty subset of $\mathcal{G}^{\prime}$, it has a positive measure
%with respect to the Haar measure on
%$\mathcal{G}^{\prime}$.
%\end{rem}

\section{Frobenius elements}
\label{Felement}

Let $\P$ be a finite nonempty set of primes.
 Let $K$ be a number field and  $L \subset \bar{K}$  a Galois extension of $K$  that is unramified outside a finite set of places of $K$.
Let $\mathcal{G}:=\Gal(L/K)$ be the Galois group of $L/K$.  

 Let $v$ be a  nonarchimedean place of $K$. Let us choose an extension $\bar{v}$ of $v$ to $\bar{K}$. Let $D(\bar{v})\subset \Gal(K)$ be the decomposition group of  
$\bar{v}$  and
$I(\bar{v})\subset D(\bar{v})$ the (normal)  inertia (sub)group of $\bar{v}$. It is known that the quotient $D(\bar{v})/I(\bar{v})$ is canonically isomorphic to the absolute Galois group $\Gal(k(v))$ of the finite {\sl residue field} $k(v)$ at $v$. In particular, this quotient has a canonical generator $\phi_{\bar{v}}$ that corresponds to the {\sl Frobenius automorphism} in  $\Gal(k(v))$.

There is the natural continuous surjective homomorphism (restriction map)
$$\res_L: \Gal(K) \twoheadrightarrow \Gal(L/K)$$
that kills $I(\bar{v})$ if and only if $v$ is unramified in $L$. If this is the case then the restriction then  $\res_L$ induces a continuous homomorphism
$D(\bar{v})/I(\bar{v}) \to \Gal(L/K)$ and we call the image of  $\phi_{\bar{v}}$ the Frobenius element at $\bar{v}$ in $\Gal(L/K)$ and denote it
$$\Frob_{\bar{v},L} \in \Gal(L/K).$$
All the $\Frob_{\bar{v},L}$'s (for a given $v$) constitute a {\sl conjugacy class} in $\Gal(L/K)$.

If $L^{\prime}/K$ is a Galois subextension of $L/K$ then the corresponding Frobenius element
$$\Frob_{\bar{v},L^{\prime}} \in \Gal(L^{\prime}/K)$$
coincides with the image of $\Frob_{\bar{v},L}$ under the natural surjective homomorphism (restriction map)
$$\Gal(L/K) \twoheadrightarrow \Gal(L^{\prime}/K).$$

We will need the following  variant of the Chebotarev's density theorem 
that is due to Serre \cite[Ch. I, Sect. 2.2, Cor. 2]{Serre}.

\begin{lem}
\label{densityS}
Let 
$\mathcal{X}$  be a subset of the Galois group $\mathcal{G}=\Gal(L/K)$ that is stable under conjugation. Assume that the boundary of $\mathcal{X}$ has measure $0$
 with respect to the Haar measure on  $\mathcal{G}$.  Then the set
of nonarchimedean places $v$ of $K$ such that corresponding {\sl Frobenius elements} $\Frob_{\bar{v}}$ lie in $\mathcal{X}$ has {\sl positive density}.
\end{lem}

 We will
apply Lemma \ref{densityS} in the following situation. 

The field $L$ is a compositum of infinite Galois extensions $K(\ell/K)$ for all $\ell\in \P$. The inclusions $K \subset K(\ell)\subset L$ induces a continuous surjective homomorphism
$$\rho_{\ell}:\mathcal{G}=\Gal(L/K)\twoheadrightarrow \Gal(K(\ell)/K),$$
which we denote by 
$$\rho_{\ell}: \mathcal{G}\twoheadrightarrow \Gal(K(\ell)/K).$$
The product-homomorphism
$$\rho: \mathcal{G}\to \prod_{\ell\in \P} \Gal(K(\ell)/K),  \ \sigma\mapsto \{\rho_{\ell}(\sigma)\}_{\ell\in \P}$$
is an embedding, whose (homeomorphic)  image is 
 a certain  closed subgroup of the product $\prod_{\ell\in \P} \Gal(K(\ell)/K)$ that maps surjectively on each factor.
Further we will identify $\mathcal{G}$ with this closed subgroup in $\prod_{\ell\in \P} \Gal(K(\ell)/K)$.

\begin{lem}
\label{density}
 Suppose that for each $\ell\in \P$ we are given the following data.
\begin{itemize}
\item
A $\Q_{\ell}$-vector space $V_{\ell}$ of finite positive dimension $d_{\ell}$ provided with a $\Z_{\ell}$-lattice $T_{\ell}\subset V_{\ell}$ of rank $d_{\ell}$.
\item
A connected reductiive linear algebraic subgroup $\GG_{\ell}\subset \GL(V_{\ell})$ of positive dimension.
\item
An element
$$\ub_{\ell} \in\GG_{\ell}(\Q_{\ell}) \subset \Aut_{\Q_{\ell}}(V_{\ell})$$
such that its characteristic polynomial
$P_{\ub_{\ell}}(t)=\det(t\II-\ub_{\ell},V) \in \Q_{\ell}[t]$
has no multiple roots. 
\item
A compact subgroup
$$G_{\ell} \subset \Aut_{\Z_{\ell}}(T_{\ell})\subset \Aut_{\Q_{\ell}}(V_{\ell})$$
that is   an open subgroup in $\GG_{\ell}(\Q_{\ell})$.
\item
An isomorphism of compact groups
$$\Gal(K(\ell)/K) \cong G_{\ell}.$$
Further we will identify these two groups via this isomorphism and $\mathcal{G}$ with the certain closed subgroup of
 $\prod_{\ell\in \P} G_{\ell}$ that maps surjectively on each factor. We keep the notation $\rho_{\ell}$ for the projection map
$$\mathcal{G}\twoheadrightarrow G_{\ell}.$$
We have for each $\ell \in \P$ and $\sigma \in \mathcal{G}$
$$\rho_{\ell}(\sigma)\in G_{\ell}  \subset \Aut_{\Z_{\ell}}(T_{\ell})\subset \Aut_{\Q_{\ell}}(V_{\ell}).$$
\item
For each $\ell \in \P$ let us consider the subset $Y_{\ell}  \subset \mathcal{G}$ that consists of all $\sigma \in \mathcal{G}$ such that the centralizer $\ZZ(\rho_{\ell}(\sigma))_0$ of 
$\rho_{\ell}(\sigma)$ in $\End_{\Z_{\ell}}(T_{\ell})\subset \End_{\Q_{\ell}}(V_{\ell})$ is isomorphic (as a $\Z_{\ell}$-algebra) to $\ZZ(\ub_{\ell})_0$.  
\end{itemize}
Let us consider the intersection 
$$Y=\bigcap_{\ell\in P}Y_{\ell}\subset \mathcal{G} \subset \prod_{\ell\in P} G_{\ell}.$$
Then  the set
of nonarchimedean places $v$ of $K$ such that corresponding {\sl Frobenius elements} $\Frob_{\bar{v}}$ lie in $Y$ has density $>0$.
\end{lem}

\begin{proof}
%We know (Corollary \ref{manyELL}) that $\mathcal{G}$ is an open subgroup of finite index in $\prod_{\ell\in P} G_{\ell}$.
Let us put $\mathcal{X}:=Y\subset \mathcal{G}$. We know that $Y$ is stable under conjugation,  has positive measure,  and its boundary has measure $0$
with respect to the Haar measure on $\mathcal{G}$ (Remark \ref{positive} and Corollary \ref{manyELL}). Now the result follows from Lemma \ref{densityS}.
\end{proof}

\begin{rem}
\label{indexD}
Suppose that for each $\ell \in \P$ we are given an open {\sl normal} subgroup $G_{\ell}^{\prime}$ in $G_{\ell}$
of finite index. Let us put
$$\mathcal{G}^{\prime}=\mathcal{G}\bigcap \prod_{\ell\in P} G_{\ell}^{\prime}\subset  \mathcal{G}, 
\ Y^{\prime}=Y\bigcap \mathcal{G}^{\prime}\subset \mathcal{G}^{\prime}.$$
Then $\mathcal{G}^{\prime}$ is  open subgroup of finite index in $\mathcal{G}$ and therefore is closed in $\mathcal{G}$.
We know that $Y$ is open in $\mathcal{G}$ and its boundary contains the identity element. This implies that $Y^{\prime}$ is an open nonempty subset
of $\mathcal{G}$; in particular, it has positive measure with respect to the Haar measure on $\mathcal{G}$.
  Since each $G_{\ell}^{\prime}$ is normal in $G_{\ell}$, the subgroup $\prod_{\ell\in P} G_{\ell}^{\prime}$ is normal
in $\prod_{\ell\in P} G_{\ell}$ and therefore $\mathcal{G}^{\prime}$ is normal in $\mathcal{G}$, which implies that $Y^{\prime}$ is a subset of 
$\mathcal{G}$ that is stable under conjugation.
  On the other hand, the boundary of $Y^{\prime}$ lies in the boundary of $Y$ and therefore also has measure zero with respect to the Haar measure
on $\mathcal{G}$.  Now  Lemma \ref{densityS} implies that the set
of nonarchimedean places $v$ of $K$ such that corresponding {\sl Frobenius elements} $\Frob_{\bar{v}}$ lie in $Y^{\prime}$ has density $>0$.

\end{rem}

\begin{sect}
\label{remind}
Let $\P$ be a nonempty finite set of primes, $A$ an abelian variety of positive dimension $g$ over a number field $K$.
Let us put 
$$ d=2g,  V_{\ell}=V_{\ell}(A), T_{\ell}=T_{\ell}(A), \rho_{\ell}=\rho_{\ell,A},$$
$$\mathcal{G}=\Gal(K), \ G_{\ell}= \rho_{\ell,A}(\Gal(K))=G_{\ell,A}.$$
We define $K(\ell)\subset \bar{K}$ as the field $\bigcup_{i=1}^{\infty}K(A_{\ell^i})$
of definition of all $\ell$-power torsion points on $A$.
It follows from the definition of Tate modules that $K(\ell)$ coincides with the subfield of $\ker(\rho_{\ell,A})$-invariants
in $\bar{K}$ and $\Gal(K(\ell)/K)=G_{\ell,A}$.
Let $v$ be a nonarchimedean place of $K$ and $\bar{v}$ an extension of $v$ to $\bar{K}$. Assume that $A$ has good reduction at $v$ and the residual
chacacteristic od $v$ is different from $\ell$.  Then
$$\Frob_{\bar{v},K(\ell)}=\Frob_{\bar{v},A,\ell} \in G_{\ell,A}=\Gal(K(\ell)/K) $$
(\cite[Sect. 2]{SerreTate}, \cite[Ch. I]{Serre}).   On the other hand,  recall (Sect. \ref{numberField}) that there is
an isomorphism of $\Z_{\ell}$-algebras
 $$\ZZ(\Frob_{\bar{v},A,\ell})_0 \cong \End(A(v))\otimes \Z_{\ell}. \eqno{(***)}$$
\end{sect}

\begin{thm}
\label{GSP}
Let $g$ be a positive integer.
Let $\P$ be a nonempty finite set of primes. Suppose that for every $\ell \in \P$ we are given the following data.

\begin{itemize}
\item
A $2g$-dimensional $\Q_{\ell}$-vector space $V_{\ell}$ provided with alternating nondegenerate $\Z_{\ell}$-bilinear form
$$e_{\ell}: V_{\ell} \times V_{\ell} \to \Q_{\ell}.$$
We write $\Gp(V_{\ell},e_{\ell})\subset \Aut_{\Q_{\ell}}(V_{\ell})$ for the corresponding group of symplectic similitudes.
\item
An element
$$\ub_{\ell} \in \Gp(V_{\ell},e_{\ell})\subset \Aut_{\Q_{\ell}}(V_{\ell})$$
such that the characteristic polynomial
$$P_{\ub_{\ell}}(t)=\det(t\II-\ub_{\ell}, V_{\ell}) \in \Q_{\ell}[t]$$
has no multiple roots.  Let  $\ZZ(\ub_{\ell})$ be the centralizer of $\ub_{\ell}$ in $\End_{\Q_{\ell}}(V_{\ell})$, 
which  is a commutative semisimple $\Q_{\ell}$-algebra
of dimension $2g$.
\item
A $\Z_{\ell}$-lattice $T_{\ell}$ of rank $2g$ in $V_{\ell}$ such that the restriction of
$e_{\ell}$ to $T_{\ell} \times T_{\ell}$ takes on values in $\Z_{\ell}$ and  the corresponding alternating 
$\Z_{\ell}$-bilinear form
$$T_{\ell} \times T_{\ell} \to \Z_{\ell}, \ x,y \mapsto e_{\ell}(x,y)$$
is perfect.
Let  $\ZZ(\ub_{\ell})_0$ for the centralizer of $\ub_{\ell}$ in
$$\End_{\Z_{\ell}}(T_{\ell})\subset \End_{\Q_{\ell}}(V_{\ell}),$$
which is an order in $\ZZ(\ub_{\ell})$ and coincides with the intersection $\ZZ(\ub_{\ell})\bigcap \End_{\Z_{\ell}}(T_{\ell})$.
\end{itemize}

Let $A$ be a $g$-dimensional abelian variety over a number field $K$ that admits a polarization $\lambda$ such that its degree $\deg(\lambda)$ is not
divisible by $\ell$ for all $\ell \in \P$.  Suppose that $\GG_{\ell,A}=\GSp(V_{\ell}(A),e_{\lambda,\ell}$ for all primes $\ell$.

Let $\Sigma$ be the set of nonarchimedean places of $K$ such that $A$ has good reduction at $v$, the residual characteristic of $v$ does not belong to $\P$ and the $\Z_{\ell}$-algebras $\End(A(v))\otimes \Z_{\ell}$ and  $\ZZ(\ub_{\ell})_0$  are isomorphic  for all $\ell\in \P$.

Then $\Sigma$ has positive density.
\end{thm}

\begin{rem}
If $\GG_{\ell,A}=\GSp(V_{\ell}(A),e_{\lambda,\ell})$ for one prime $\ell$ then it is true for all primes \cite{ZarhinMMJ}.
Such $A$ are sometimes called abelian varieties of GSp type. If $A$ is an abelian variety of GSp type then the set of nonarchimedean places $v$ of $K$ such
that $\End^0(A(v))$ is  a degree $2g$ CM field has density 1 \cite{Zyvina}.

\end{rem}
\begin{proof}
For each $\ell \in \P$ let us fix a symplectic isomorphism
$$\phi_{\ell}:   (T_{\ell}(A), e_{\lambda,\ell}) \cong (T_{\ell}, e_{\ell}).$$
Extending $\phi_{\ell}$ by $\Q_{\ell}$-linearity, we obtain the symplectic isomorphism
$$ (V_{\ell}(A), e_{\lambda,\ell}) \cong(V_{\ell}, e_{\ell}),$$
which we continue to denote by $\phi_{\ell}$.  Clearly,
$$\Gp(V_{\ell}(A), e_{{\lambda},\ell})= \phi_{\ell}^{-1}\Gp(V_{\ell}), e_{\ell})\phi_{\ell}.$$
Let us put
$$\ub_{\ell}^{\prime}= \phi_{\ell}^{-1}\ub_{\ell} \phi\in  \phi_{\ell}^{-1}\Gp(V_{\ell}), e_{\ell})\phi_{\ell}=Gp(V_{\ell}(A), e_{{\lambda},\ell})\subset \Aut_{\Q_{\ell}}(V_{\ell}(A).$$
Clearly, the characteristic polynomial of $\ub_{\ell}^{\prime}$ has no multiple roots  (since it coincides with characteristic polynomial of $\ub_{\ell}$) and the centralizer
$\ZZ(\ub_{\ell}^{\prime})_0$  is isomorphic as $\Z_{\ell}$-algebra to $\ZZ(\ub_{\ell})_0$.

 Now Theorem \ref{GSP} follows from Lemma \ref{density} combined with (***).

\end{proof}

\section{Proof of main results}
\label{mainProof}
%Let $B$ be an abelian variety

\begin{proof}[Proof of Theorem \ref{main}]

In light of Sect. \ref{remind}, the result follows from Lemma \ref{density} combined with (***).
\end{proof}

\begin{proof}[Proof of Theorem \ref{jacobian}]
Recall  that the  $A$ is a {\sl jacobian} and therefore admits a canonical  principal polarization ${\lambda}$. This implies that the corresponding
alternating $\Z_{\ell}$-bilinear form
$$e_{{\lambda},\ell}: T_{\ell}(A) \times T_{\ell}(A) \to \Z_{\ell}$$
is unimodular.  It is also known \cite{ZarhinMMJ} that  our assuptions on the Galois group of  $f(x)$ imply that
$$\GG_{A,\ell}=\GSp(V_{\ell}(A), e_{\lambda,\ell})$$ for all primes $\ell$. 

For each $\ell \in P$ the abelian variety$B^{(\ell)}$  admits a polarization say, $\mu_{\ell}$ of degree prime to $\ell$.
This implies that the corresponding
alternating $\Z_{\ell}$-bilinear form
$$e_{\mu_{\ell},\ell}: T_{\ell}(B^{(\ell)}) \times T_{\ell}(B^{(\ell)}) \to \Z_{\ell}$$
is unimodular.  Let us put
$$V_{\ell} =  V_{\ell}(B^{(\ell)}),\  T_{\ell}=T_{\ell}(B^{(\ell)}),\  e_{\ell}=e_{\mu_{\ell},\ell}.$$

  Since both alternating  forms $e_{{\lambda},\ell}$ and $e_{\mu,\ell}$ are unimodular and the ranks of free $\Z_{\ell}$-modules
$T_{\ell}(A)$ and $T_{\ell}(B^{(\ell)})$ do coincide, there is a symplectic isomorphism of free $\Z_{\ell}$-modules
$$\phi_{\ell}i:T_{\ell}(A) \cong T_{\ell}(B^{(\ell)}),$$
which extends by $\Q_{\ell}$-linearity to the symplectic isomorphism of $\Q_{\ell}$-vector spaces
$$V_{\ell}(A) \cong V_{\ell}(B^{(\ell)}),$$
which we continue to denote $\phi$. Clearly,
$$\Gp(V_{\ell}(A), e_{{\lambda},\ell})= \phi^{-1}\Gp(V_{\ell}(B^{(\ell)}), e_{\mu,\ell})\phi.$$
Using Theorem \ref{CM}, pick
$$\ub_{\ell}\in  \End((B^{(\ell)}))\subset  \End_{\Q_{\ell}}(V_{\ell}(B^{(\ell)})=\End_{\Q_{\ell}}(V_{\ell})$$ 
such that its characteristic polynomial has no multiple roots,
$\ub_{\ell}$ lies in $$\Gp(V_{\ell}(B^{(\ell)}),e_{\mu,\ell})=\Gp(V_{\ell},e_{\ell})$$ and the centralizer
$\ZZ(\ub_{\ell})_0$ of $\ub_{\ell}$ in $\End_{\Z_{\ell}}(T_{\ell}(B^{(\ell)}))=\End_{\Z_{\ell}}(T_{\ell})$
 coincides with $\End((B^{(\ell)}))\otimes \Z_{\ell}$.
Now the result follows from Theorem \ref{GSP}.
%Let us put
%$$\ub_{\ell}= \phi^{-1}u_{\ell} \phi\in  \phi^{-1}\Gp(V_{\ell}(B^{(\ell)}), e_{\mu,\ell})\phi=Gp(V_{\ell}(A), e_{{\lambda},\ell})\subset \Aut_{\Q_{\ell}}(V_{\ell}(A).$$
%Clearly, the characteristic polynomial of $\ub_{\ell}$ has no multiple roots  (since it coincides with characteristic polynomial of $u_{\ell}$) and the centralizer
%$\ZZ(\ub_{\ell})_0$  is isomorphic as $\Z_{\ell}$-algebra to $\End((B^{(\ell)})\otimes \Z_{\ell}$.

% Now Theorem \ref{jacobian} follows from Lemma \ref{density} combined with (***).
\end{proof}

\section{Complements}

\begin{thm}
\label{jacobian2}
Let $g \ge 2$ be an integer, $n=2g+1$ or $2g+2$.  Let $\P$ be a nonempty finite set of primes and suppose
that for each $\ell \in \P$  we are given a $g$-dimensional semisimple commutative $\Q_{\ell}$-algebra $\CC_{\ell}$
and an order $\RR_{\ell}$ in $\CC_{\ell}$.

 Let $K$ be a number field and $f(x) \in K[x]$ be a degree $n$
 irreducible polynomial, whose Galois group over $K$ is either full symmetric group $\ST_n$ or the alternating group $\A_n$.
Let us consider the genus $g$ hyperelliptic curve $C_f:y^2=f(x)$ and its  jacobian $A$, which is a $g$-dimensional abelian variety over $K$.

Let $\Sigma$ be the set of all nonarchimedean places $v$ of $K$ such that $A$ has good reduction at $v$, the residual characteristic $\fchar(k(v))$ does not 
belong to $\P$ and the $\Z_{\ell}$-rings $\End(A)\otimes \Z_{\ell}$ and $\RR_{\ell}\oplus \RR_{\ell}$ are isomorphic for all $\ell\in \P$.  Then $\Sigma$ has density $>0$.
\end{thm}

\begin{proof}
Recall that $A$ admits a principal polarization $\lambda$ and for each prime $\ell$ 
$$e_{\lambda,\ell}: T_{\ell}(A)\times  T_{\ell}(A) \to \Z_{\ell}$$
is the corresponding alternating perfect $\Z_{\ell}$-bilinear pairing.
Let   $\ell$ be a prime that lies in $\P$. Let us put 
$$\RR=\RR_{\ell}, \ \CC=\CC_{\ell}$$ and 
fix a free $\RR=\RR_{\ell}$-module $\mathcal{T}=\mathcal{T}_{\ell}$ of rank 1 (e.g.,  $\mathcal{T}_{\ell}=\RR_{\ell}$).
Let 
$$\mathcal{V}=\mathcal{V}_{\ell}, \mathbf{T}= \mathbf{T}_{\ell}, \mathbf{V}= \mathbf{V}_{\ell}$$
be as in Example \ref{symplecticSplit}.  In particular,  $\mathbf{T}_{\ell}$ is a free $\Z_{\ell}$-module of rank $2g$ that
is a lattice in the $2g$-dimensional $\Q_{\ell}$-vector space $\mathbf{V}$.

In addition, using  Example \ref{symplecticSplit}, we obtain an alternating perfect $\Z_{\ell}$-bilinear form
$$e_{\ell}: \mathbf{T}_{\ell} \times \mathbf{T}_{\ell}  \to \Z_{\ell}$$
and an element
$$\ub_{\ell} \in \Gp(\mathbf{V}_{\ell},e_{\ell})\subset \Aut_{\Q_{\ell}}(\mathbf{V}_{\ell})$$
such that the centralizer $\ZZ(\ub_{\ell})_0$ in $\End_{\Z_{\ell}}(\mathbf{T}_{\ell})\subset \End_{\Q_{\ell}}(\mathbf{V}_{\ell}) $ is isomorphic to $\RR_{\ell}\oplus \RR_{\ell}$.

Now the result follows from Theorem \ref{remind}.

% Since both alternating  forms $e_{{\lambda},\ell}$ and $e_{\ell}$ are unimodular and the ranks of free $\Z_{\ell}$-modules
%$T_{\ell}(A)$ and $\mathbf{T}_{\ell}$ do coincide, there is a symplectic isomorphism of free $\Z_{\ell}$-modules
%$$\phi_{\ell}:T_{\ell}(A) \cong \mathbf{T}_{\ell},$$
%which extends by $\Q_{\ell}$-linearity to the symplectic isomorphism of $\Q_{\ell}$-vector spaces
%$$V_{\ell}(A) \cong  \mathbf{V}_{\ell},$$
%which we continue to denote $\phi$. Clearly,
%$$\Gp(V_{\ell}(A), e_{{\lambda},\ell})= \phi_{\ell}^{-1}\Gp(\mathbf{V}_{\ell}), e_{\ell})\phi_{\ell}.$$

%Let us put
%$$\ub_{\ell}^{\prime}= \phi_{\ell}^{-1}\ub_{\ell} \phi_{\ell}\in  \phi_{\ell}^{-1}\Gp(\mathbf{V}_{\ell}), e_{\ell})\phi_{\ell}=Gp(V_{\ell}(A), e_{{\lambda},\ell})\subset %\Aut_{\Q_{\ell}}(V_{\ell}(A)).$$
%Clearly, the characteristic polynomial of $\ub_{\ell}^{\prime}$ has no multiple roots  (since it coincides with characteristic polynomial of $\ub_{\ell}$) and the centralizer
%$\ZZ(\ub_{\ell}^{\prime})_0$  is isomorphic as $\Z_{\ell}$-algebra to $\RR_{\ell}\oplus \RR_{\ell}$.

% Now Theorem \ref{jacobian2} follows from Lemma \ref{density} combined with (***).

\end{proof}

\begin{thm}
\label{jacobian3}
Let $g \ge 2$ be an integer, $n=2g+1$ or $2g+2$.  Let $\P$ be a nonempty finite set of primes and suppose
that for each $\ell \in \P$  we are given the following data.

\begin{itemize} 
\item
A degree $g$ field extension $F_{0,\ell}/\Q_{\ell}$. We write $\Oc_{0,\ell}$ for the ring of integers in the $\ell$-adic field  $F_{0,\ell}$.
\item
A $2$-dimensional semisimple commutative $\F_{0,\ell}$-algebra $\CC_{\ell}$.
\item
An $\Oc_{0,\ell}$-subalgebra  $R_{\ell}$ of $\CC_{\ell}$ that is a free $\Oc_{0,\ell}$-module of rank $2$.
\end{itemize}

 Let $K$ be a number field and $f(x) \in K[x]$ be a degree $n$
 irreducible polynomial, whose Galois group over $K$ is either full symmetric group $\ST_n$ or the alternating group $\A_n$.
Let us consider the genus $g$ hyperelliptic curve $C_f:y^2=f(x)$ and its  jacobian $A$, which is a $g$-dimensional abelian variety over $K$.

Let $\Sigma$ be the set of all nonarchimedean places $v$ of $K$ such that $A$ has good reduction at $v$, the residual characteristic $\fchar(k(v))$ does not 
belong to $\P$ and the $\Z_{\ell}$-rings $\End(A)\otimes \Z_{\ell}$ and $R_{\ell}$ are isomorphic for all $\ell\in \P$.  Then $\Sigma$ has density $>0$.
\end{thm}

\begin{proof}
The proof is literally the same as the proof of Theorem \ref{jacobian2} with the only modification: we need to use Lemma \ref{generator2} instead of Example
\ref{symplecticSplit}.
\end{proof}

\begin{rem}
 Let ${\bf N}$ be a positive integer. The assertions of Theorems \ref{main},\ref{jacobian}, \ref{GSP}, \ref{jacobian2}, \ref{jacobian3}
(resp. of Example \ref{elliptic} and Corollary \ref{discrE})
remain true if
we impose an additional condition on the places $v$ that the residual characteristic of $v$ does not divide 
${\bf N}$ and $A(v)_{{\bf N}}$ lies in $A(k(v))$ (resp. $E(v)_{{\bf N}}$ lies in $E(k(v))$). Indeed, let $\P^{\prime}$ be the set of prime
divisors of ${{\bf N}}$. Then the proofs remain the same with the only modification: we should deal with the finite set of primes
$\tilde{\P}=\P \bigcup \P^{\prime}$ (instead of $\P$) and apply Remark \ref{indexD} (instead of Lemma \ref{density})
to $G_{\ell}=G_{\ell,A}$ for all $\ell \in \tilde{\P}$,
$$G_{\ell}^{\prime}=G_{\ell,A} \bigcap [\II+ {\bf N}\cdot \End_{\Z_{\ell}}(T_{\ell}(A))]\subset G_{\ell,A}=G_{\ell}$$
if $\ell\mid {\bf N}$ and 
$$G_{\ell,A}^{\prime}=G_{\ell,A}=G_{\ell}$$ if  $\ell$ does {\sl not} divide ${\bf N}$.
\end{rem}

\begin{rem}
In  Theorems \ref{main},\ref{jacobian}, \ref{GSP}, \ref{jacobian2}, \ref{jacobian3} we assume that $\Gal(f)=\ST_n$ or $\A_n$  only in order
to make sure that the jacobian is of GSp type. See \cite{ZarhinMMJ,ZarhinPLMS,ZarhinTAMS} where we discuss the cases of smaller $\Gal(f)$'s when 
 the jacobian is still of GSp type and therefore Theorems \ref{main},\ref{jacobian}, \ref{GSP}, \ref{jacobian2}, \ref{jacobian3} remain true.
\end{rem}

\end{document}